\documentclass{amsart}

\usepackage{amsmath}
\usepackage{amssymb}
\usepackage{amsthm}
\usepackage{amsfonts}
\usepackage{calligra}
\usepackage{appendix}
\usepackage{graphicx}
\usepackage{enumerate}
\usepackage{mathrsfs}
\usepackage[english]{babel}
\usepackage{mathtools}
\usepackage{hyperref}

\def\Aint{-\nobreak \hskip-10.8pt \nobreak\int}
\def\R{\mathbb R}
\def\N{\mathbb N}
\def\Z{\mathbb{Z}}
\def\NN{\mathcal{N}}
\def\W{\mathcal{W}}

\def\u{\mathbf{u}}
\def\vv{\mathbf{v}}
\def\x{\mathbf{x}}
\def\tp{\textup}
\def\NN{\mathcal{N}}
\def\w{\mathbf{w}}
\def\F{\mathbf{F}}
\def\S{\mathcal{S}}
\def\T{\mathbb{T}}
\def\0{\mathbf{0}}

\newcommand{\norm}[1]{\left\lVert#1\right\rVert}
\newcommand{\abs}[1]{\left\lvert#1\right\rvert}

\newcommand{\defeq}{\mathrel{:\mkern-0.25mu=}}
\newcommand{\eqdef}{\mathrel{=\mkern-0.25mu:}}

\newtheorem{theorem}{Theorem}[section]
\newtheorem{lemma}[theorem]{Lemma}
\newtheorem{proposition}[theorem]{Proposition}

\newtheorem{corollary}[theorem]{Corollary}
\newtheorem{claim}[theorem]{Claim}

\theoremstyle{definition}
\newtheorem{definition}[theorem]{Definition}

\theoremstyle{remark}
\newtheorem{remark}[theorem]{Remark}

\title[Integrability properties of Leray-Hopf solutions]{On the integrability properties of Leray-Hopf solutions of the Navier-Stokes equations on $\R^3$}

\begin{document}

\author{Sauli Lindberg}
\address{Department of Mathematics and Statistics, University of Helsinki, P.O. Box 68, 00014
Helsingin yliopisto, Finland}
\curraddr{}
\email{sauli.lindberg@helsinki.fi}
\thanks{The author acknowledges support by the ERC Advanced Grant 834728 and the Academy of Finland CoE FiRST}

\begin{abstract}
Let $r,s \in [2,\infty]$ and consider the Navier-Stokes equations on $\R^3$. We study the following two questions for suitable $s$-homogeneous Banach spaces $X \subset \S'$: does every $\u_0 \in L^2_\sigma$ have a weak solution that belongs to $L^r(0,\infty;X)$, and are the $L^r(0,\infty;X)$ norms of the solutions bounded uniformly in viscosity? We show that if $\frac{2}{r} + \frac{3}{s} < \frac{3}{2}-\frac{1}{2r}$, then for a Baire generic datum $\u_0 \in L^2_\sigma$, no weak solution $\u^\nu$ belongs to $L^r(0,\infty;X)$. If $\frac{3}{2}-\frac{1}{2r} \leq \frac{2}{r} + \frac{3}{s} < \frac{3}{2}$ instead, global solvability in $L^r(0,\infty;X)$ is equivalent to the \emph{a priori} estimate $\norm{\u^\nu}_{L^r(0,\infty;X)} \leq C \nu^{3-5/r-6/s} \norm{\u_0}_{L^2}^{4/r+6/s-2}$. Furthermore, we can only have $\limsup_{\nu \to 0} \norm{\u^\nu}_{L^r(0,\infty;Z)} < \infty$ for all $\u_0 \in L^2_\sigma$ if $\frac{2}{r} + \frac{3}{s}= \frac{3}{2}-\frac{1}{2r}$.

The above results and their variants rule out, for a Baire generic $L^2_\sigma$ datum, $L^4(0,T;L^4)$ integrability and various other known sufficient conditions for the energy equality. As another application, for suitable 2-homogeneous Banach spaces $Z \hookrightarrow L^2_\sigma$, each $\u_0 \in Z$ has a Leray-Hopf solution $\u \in L^3(0,\infty;\dot{B}_{3,\infty}^{1/3})$ if and only if a uniform-in-viscosity bound $\norm{\u}_{L^3(0,\infty;\dot{B}_{3,\infty}^{1/3})}^3 \leq C \norm{\u_0}_Z^2$ holds.

As a by-product we show that if global regularity holds for the Navier-Stokes equations, then for a Baire generic $L^2_\sigma$ datum, the Leray-Hopf solution is unique and satisfies the energy equality. We also show that if global regularity holds in the Euler equations, then anomalous energy dissipation must fail for a Baire generic $L^2_\sigma$ datum. These two results also hold on the torus $\T^3$.
\end{abstract}

\maketitle

\section{Introduction}
The aim of this paper is to clarify some relations between Leray-Hopf solutions of the 3D incompressible Navier-Stokes equations and integrability conditions which are known to guarantee smoothness, uniqueness or the energy equality. Most of the results either rule out an integrability class or give a necessary condition for integrability in terms of an \emph{a priori} estimate, taking advantage of techniques introduced by Guerra, Koch and the author in~\cite{GKL23}. As a by-product, we also prove some necessary conditions for global regularity in the Navier-Stokes and Euler equations.

\subsection{Background} \label{ss:Background}
We start by recalling some relevant notions and classical results on the Navier-Stokes equations (see e.g.~\cite{Gal00,RRS16}). We consider the Cauchy problem
\begin{align}
& \partial_t \u + (\u \cdot \nabla) \u - \nu \Delta \u + \nabla p = 0, \qquad \nabla \cdot \u = 0, \label{e:NS1}\\
& \u(\cdot,0) = \u_0, \label{e:NS2}
\end{align}
where $\u$ is the velocity field, $p$ is the pressure, $\nu > 0$ is the kinematic viscosity and the initial datum $\u_0$ is solenoidal. We will work on $\R^3$ unless otherwise specified.

Given an initial datum $\u_0 \in L^2_{\tp{loc}} = L^2_{\tp{loc}}(\R^3)$ with $\nabla \cdot \u_0 = 0$, a vector field $\u \in L^2_{\tp{loc}}([0,T) \times \R^3)$ is called a \emph{very weak solution} (or \emph{distributional solution}) of \eqref{e:NS1}--\eqref{e:NS2} if $\nabla \cdot \u(\cdot,t) = 0$ a.e. $t \in (0,T)$ and
\begin{equation} \label{e:Weak solution}
\int_0^T \int_{\R^3} [\u \cdot (\partial_t \varphi + \nu \Delta \varphi) + \u \otimes \u \colon \nabla \varphi] \, dx \, dt = - \int_{\R^3} \u_0 \cdot \varphi(\cdot,0) \, dx
\end{equation}
for all $\varphi \in C_c^\infty([0,T) \times \R^3)$ with $\nabla \cdot \varphi = 0$. (If $\nu = 0$ in \eqref{e:Weak solution}, then $\u$ is called a \emph{weak solution of the Euler equations}.) If $\u_0 \in L^2_\sigma$ and $\u \in L^\infty(0,T;L^2_\sigma) \cap L^2(0,T;\dot{H}^1)$, then $\u$ is called a \emph{weak solution} of \eqref{e:NS1}--\eqref{e:NS2}. A weak solution can be redefined on a set of times of measure zero to get $\u \in C_w([0,T);L^2_\sigma)$, and in particular $\u(t) \rightharpoonup \u_0$ in $L^2_\sigma$ as $t \to 0$. We will henceforth use this weakly continuous representative.

A weak solution $\u$ on $\R^3 \times [0,\infty)$ is called a \emph{Leray-Hopf solution} if it satisfies the \emph{energy inequality}
\begin{equation} \label{e:Energy inequality}
\frac{1}{2} \norm{\u(t)}_{L^2}^2 + \nu \int_0^t \norm{\nabla \u(\tau)}_{L^2}^2 d\tau \leq \frac{1}{2} \norm{\u_0}_{L^2}^2
\end{equation}
at every $t > 0$. It is open whether every weak solution is a Leray-Hopf solution. The existence of a Leray-Hopf solution for every $\u_0 \in L^2_\sigma$ was proved by Leray~\cite{Ler34} on $\R^3$ and by Hopf~\cite{Hop51} in all domains $\Omega \subset \R^3$. Given $\u_0 \in L^2_\sigma$ we denote
\begin{align*}
& \mathcal{N}_\nu(\u_0) \defeq \{\text{Leray-Hopf solutions of \eqref{e:NS1}--\eqref{e:NS2}}\}, \\
& \mathcal{W}_\nu(\u_0) \defeq \{\text{weak solutions of \eqref{e:NS1}--\eqref{e:NS2}}\}.
\end{align*}
Every Leray-Hopf solution satisfies $\lim_{t \to 0} \norm{\u(t)-\u_0}_{L^2} = 0$.

By interpolation, weak solutions of \eqref{e:NS1}--\eqref{e:NS2} belong to $L^r(0,\infty;L^s)$ whenever $s \in [2,6]$ and
\begin{equation} \label{e:Stokes scaling}
\frac{2}{r} + \frac{3}{s} = \frac{3}{2}.
\end{equation}
In the case of the Euler equations, the correct scaling is $\frac{2}{r} + \frac{3}{s} = \frac{3}{2}-\frac{1}{2r}$, that is,
\begin{equation} \label{e:Euler scaling}
\frac{5}{3r} + \frac{2}{s} = 1,
\end{equation}
in the sense that if \eqref{e:Euler scaling} fails, then a weak solution $\u \in L^r(0,\infty;L^s)$ only exists for a meagre set of data $\u_0 \in L^2_\sigma$; see Theorem \ref{t:Euler}. The result on the Euler equations does not seem to have appeared in the literature, but it follows directly from~\cite[Theorem 3.8]{GKL23}. In \textsection \ref{s:The proof of generic non-solvability theorem} we present a short, self-contained proof. The main idea is to relax the Euler equations into a linear system and thereby generate new scaling symmetries, which allows freer use of scaling arguments.

It is open whether Leray-Hopf solutions are unique, belong to $C^\infty((0,\infty) \times \R^3)$ and/or satisfy the \emph{energy equality} (equality sign in \eqref{e:Energy inequality}). In two recent breakthroughs, Buckmaster and Vicol showed the non-uniqueness of $C([0,T];H^\beta(\T^3))$ mild solutions for a small $\beta > 0$~\cite{BV19} and Albritton, Bru\'e and Colombo showed the non-uniqueness of Leray-Hopf solutions on $\R^3$ under a suitable external forcing $\F \in L^1(0,T;L^2)$ and with $\u_0 = 0$~\cite{ABC}.

If, however, a Leray-Hopf solution $\u$ belongs to $L^r(0,T;L^s)$, where $r$ and $s$ satisfy an \emph{LPS (Ladyzhenskaya-Prodi-Serrin) condition}
\begin{equation} \label{e:Prodi-Serrin}
\frac{2}{r} + \frac{3}{s} = 1, \qquad s \in [3,\infty],
\end{equation}
then $\u$ is well-known to be unique and smooth in $\R^3 \times (0,T]$ and satisfy the energy equality. For the smoothness see e.g.~\cite[Theorem 8.17]{RRS16} when $s > 3$; the case $s = 3$ is much harder and was proved by Escauriaza, Seregin and \v{S}ver\'ak~\cite{ISS03}. 
A classical weak-strong uniqueness statement under LPS conditions is formulated in Theorem \ref{t:Weak-strong uniqueness}. Cheskidov and Luo have shown that for every $d \geq 2$ and $p \in [1,2)$, uniqueness fails for very weak solutions in $L^p(0,T;L^\infty(\T^d))$, even for smooth initial data (uniqueness holds in the LPS class $L^2(0,T;L^\infty(\T^d)$)~\cite{CL22}.

Leray-Hopf solutions generally fail to satisfy LPS conditions. More precisely, if $\u_0 \in L^2_\sigma$ and \eqref{e:Prodi-Serrin} holds with $3 \leq s < \infty$, then there exists $\u \in \mathcal{N}_\nu(\u_0) \cap \cup_{T>0} L^r(0,T;L^s)$ if and only if $e^{t\nu\Delta} \u_0 \in \cup_{T > 0} L^r(0,T;L^s)$ (if and only if $\u_0$ belongs to the Besov space $\dot{B}_{s,r}^{-2/r}$). This seems to be a folklore result, but in bounded $C^{2,1}$ smooth domains it is due to Farwig, Sohr and Varnhorn~\cite{FS09,FSV09}; see also~\cite{FGH16} for a version on critical time-weighted LPS-type classes and~\cite{LR22} for more general spaces of initial data.

The energy equality is, however, also known to hold under weaker assumptions than \eqref{e:Prodi-Serrin}. The case $\u \in \W_\nu(\u_0) \cap L^4(0,T;L^4)$ is due to Lions~\cite{Lio60}, and Shinbrot extended the result to $\frac{2}{r} + \frac{2}{s} \leq 1$, $s \geq 4$  and all dimensions $d \geq 2$~\cite{Shi74}. We will pay special attention to another regularity class that guarantees the energy equality for weak solutions, namely $L^3(0,T;B_{3,\infty}^{1/3})$~\cite[Theorem 6.1]{CCFS08}. We have $L^\infty(0,T;L^2) \cap L^3(0,T;\dot{B}_{3,\infty}^{1/3}) \hookrightarrow L^3(0,T;B_{3,\infty}^{1/3})$ (see Lemma \ref{l:Interpolation}). The Besov space $\dot{B}_{3,\infty}^{1/3}$ is $\frac{9}{2}$-homogeneous, and therefore $L^3(0,\infty;\dot{B}_{3,\infty}^{1/3})$ satisfies \eqref{e:Euler scaling}.
More generally, when $1 \leq r \leq 3$, solutions of the Navier-Stokes equations in $L^r(0,\infty;\dot{B}_{r,\infty}^{\beta_r})$,
\begin{equation} \label{e:Besov exponent}
\beta_r \defeq \frac{11-3r}{2r},
\end{equation}
satisfy \eqref{e:Euler scaling} and the energy equality~\cite[Theorem 1.3]{CL20}. For other recently discovered energy conservation criteria see, for instance,~\cite{BC20,CL20,LS18} and the references contained therein.

To the author's knowledge it has remained open whether every Leray-Hopf solution belongs to $\cup_{T>0} L^4(0,T;L^4)$. (The argument at~\cite[p. 641]{FS09} rules out $L^r(0,T;L^s)$ integrability for generic Leray-Hopf solutions whenever $\frac{2}{r} + \frac{3}{s}<\frac{5}{4}$; see Proposition \ref{p:No higher integrability}.) We give a negative answer in Theorem \ref{t:Relaxed} and Corollary \ref{c:Navier-Stokes}.

In \textsection \ref{ss:Modified solution and data spaces} we study the following more general questions: for which function spaces $X$ and $Z$ does every $\u_0 \in L^2_\sigma \cap Z$ have a weak solution $\u^\nu \in L^r(0,\infty;X)$, or $\u^\nu \in \cup_{T>0} L^r(0,T;X)$? For which $X,Z$ is $\limsup_{\nu \to 0} \inf_{\u^\nu \in \W_\nu(\u_0)} \norm{\u^\nu}_{L^r(0,\infty;\infty)} < \infty$ for all $\u_0 \in L^2_\sigma \cap Z$? The reason for considering weak solutions instead of Leray-Hopf solutions is that the class of Leray-Hopf solutions is not known to be stable under weak$^*$ convergence; see Remark \ref{r:Weak* stability}. In \textsection \ref{ss:Time-weighted LPS classes and necessary and sufficient conditions for global regularity} we study similar questions for LPS classes endowed with a time weight.

\subsection{The main result} \label{ss:Modified solution and data spaces}

In the main result of this paper, Theorem \ref{t:Relaxed}, we prove that in order to rule out a solution space $L^r(0,\infty;X)$, it is enough to rule out the \emph{a priori} estimate \eqref{e:Supercritical a priori}. Theorem \ref{t:Relaxed} also shows that for ruling out uniform-in-$\nu$ bounds in $L^r(0,\infty;X)$ it suffices to rule out the \emph{a priori} estimate \eqref{e:Uniform a priori}. In order to state Theorem \ref{t:Relaxed} we fix some terminology.
\begin{definition}
A Banach space $\{0\} \subsetneq X \hookrightarrow \S'$ is said to be \emph{$p$-homogeneous} if $\vv(\lambda^{-1} \cdot) \in X$ and $\norm{\vv(\lambda^{-1} \cdot)}_X \approx_Z \lambda^{3/p} \norm{\vv}_X$ for all $\vv \in X$ and $\lambda > 0$, and $X$ is said to be \emph{isometrically translation invariant} if $\vv(\cdot-x_0) \in X$ and $\norm{\vv(\cdot-x_0)}_X = \norm{\vv}_X$ for every $\vv \in X$ and $x_0 \in \R^3$. (Using \emph{ad hoc} terminology,) a Banach space $Y$ satisfies the \emph{Fatou property} with respect to a dual Banach space $Z$ if $\sup_{j \in \N} \norm{\vv_j}_Y = C < \infty$ and $\vv_j \overset{*}{\rightharpoonup} \vv$ in $Z$ imply that $\vv \in Y$ and $\norm{\vv}_Y \leq C$.
\end{definition}

\begin{theorem} \label{t:Relaxed}
Let $r,s \in (4/3,\infty]$ with $1 < \frac{2}{r} + \frac{3}{s} < \frac{3}{2}$. Suppose $X,Z \subset \S'$ are isometrically translation invariant Banach spaces, $X$ is $s$-homogeneous and $Z$ is $2$-homogeneous. Suppose $L^r(0,\infty;X)$ and $Z$ have the Fatou property with respect to $L^\infty(0,\infty;L^2)$ and $L^2_\sigma$, respectively. The implications
\[\text{\eqref{i:Maximum}} \quad \Leftarrow \quad \text{\eqref{i:Non-meagre}} \quad \Leftrightarrow \quad \text{\eqref{i:A priori}} \quad \Leftarrow \quad \text{\eqref{i:Uniform in nu}} \quad \Leftrightarrow \quad \text{\eqref{i:A priori uniform in nu}} \quad \Rightarrow \quad \text{\eqref{i:Exponents}}\]
hold between the claims below:
\begin{enumerate}[(i)]
\item For some $\nu > 0$, $\u^\nu \in \mathcal{W}_\nu(\u_0) \cap L^r(0,\infty;X)$ exists for a non-meagre set of data $\u_0 \in L^2_\sigma \cap Z$. \label{i:Non-meagre}

\item There exists $C > 0$ such that for every $\nu > 0$ and $\u_0 \in L^2_\sigma \cap Z$ there exists $\u^\nu \in \mathcal{W}_\nu(\u_0)$ with
\begin{align}
& \norm{\u^\nu}_{L^\infty(0,\infty;L^2)} + \sqrt{\nu} \norm{\nabla \u^\nu}_{L^2(0,\infty;L^2)} \leq C \max \left\{ \norm{\u_0}_{L^2}, \norm{\u_0}_Z \right\}, \label{e:Minimum for energy} \\
& \norm{\u^\nu}_{L^r(0,\infty;X)}
\leq C \nu^{-3 \left( \frac{5}{3r} + \frac{2}{s} - 1 \right)} \max \left\{ \norm{\u_0}_{L^2}^{2 \left( \frac{2}{r} + \frac{3}{s} - 1 \right)}, \norm{\u_0}_Z^{2 \left( \frac{2}{r} + \frac{3}{s} - 1 \right)} \right\}. \label{e:Supercritical a priori}
\end{align}
\label{i:A priori}

\item $\frac{5}{3r} + \frac{2}{s} \geq 1$.
\label{i:Maximum}

\item For a non-meagre set of data $\u_0 \in L^2_\sigma \cap Z$ we have
\[\limsup_{\nu \to 0} \inf_{\u^\nu \in \W_\nu(\u_0)} (\norm{\u^\nu}_{L^\infty(0,\infty;L^2)} + \sqrt{\nu} \norm{\nabla \u^\nu}_{L^2(0,\infty;L^2)} + \norm{\u^\nu}_{L^r(0,\infty;X)}) < \infty.\] \label{i:Uniform in nu}

\item There exist $\nu_0, C > 0$ such that for all $\nu > 0$, $\tilde{\nu} \in (0,\nu_0]$ and $\u_0 \in L^2_\sigma \cap Z$ there exists $\u^\nu \in \W_\nu(\u_0)$ with
\begin{align}
& \norm{\u^\nu}_{L^\infty(0,\infty;L^2)} + \sqrt{\nu} \norm{\nabla \u^\nu}_{L^2(0,\infty;L^2)} \leq C \max \left\{ \norm{\u_0}_{L^2}, \norm{\u_0}_Z \right\} \label{e:Minimum uniformly} \\
& \norm{\u^\nu}_{L^r(0,\infty;X)}
\leq C \left( \frac{\tilde{\nu}}{\nu} \right)^{3 \left( \frac{5}{3r} + \frac{2}{s} - 1 \right)} \max \left\{ \norm{\u_0}_{L^2}^{2 \left( \frac{2}{r} + \frac{3}{s} - 1 \right)}, \norm{\u_0}_Z^{2 \left( \frac{2}{r} + \frac{3}{s} - 1 \right)} \right\}. \label{e:Uniform a priori}
\end{align}
\label{i:A priori uniform in nu}

\item $\frac{5}{3r} + \frac{2}{s} = 1$.
\label{i:Exponents}
\end{enumerate}
\end{theorem}

The \emph{a priori} estimates \eqref{e:Supercritical a priori}--\eqref{e:Uniform a priori} are in the spirit of the works of Gallagher~\cite{Gal01}, Tao~\cite{Tao07,Tao13}, Rusin and \v{S}ver\'ak~\cite{RS11} and Jia and \v{S}ver\'ak~\cite{JS13}. Their use is at least three-fold. First, they reassure one that one cannot prove solvability in $L^r(0,\infty;X)$ by "soft methods" alone, without also proving a uniform estimate on the norms $\norm{\u^\nu}_{L^r(0,\infty;X)}$. Second, the \emph{a priori} estimates allow one to rule out plausible solution classes via scaling arguments; this aspect is treated extensively in an abstract setting in~\cite{GKL23}. Third, they sometimes allow one to infer global solvability in a function space from local solvability, by iteratively using the fact that kinetic energy is decreasing.

The energy inequality \eqref{e:Energy inequality} allows one to control the kinetic energy and cumulative energy dissipation globally in time, but no other coercive globally controlled quantities have been found in the 3D Navier-Stokes equations. Gallagher's \emph{a priori} estimate holds for critical spaces of initial data such as $L^3$ and $\dot{H}^{1/2}$, and uses profile decompositions. As a consequence of the result, global control on such critical norms is necessary for global regularity (for data in said spaces). Supercritical quantities, such as those studied in Theorem \ref{t:Relaxed}, seem unlikely to be of decisive use in proving global regularity (see, e.g.,~\cite[\textsection 3.4]{Tao08}), but some of them would imply, among other things, the energy equality and the validity of the inviscid limit (at least in a weak$^*$ sense).

The supercritical setting of Theorem \ref{t:Relaxed} allows us to find a rather elementary proof based on Baire category and a functional analytical lemma from~\cite{GKL23} (Lemma \ref{l:GKL23} in this paper). Baire category often allows one to infer from a qualitative condition in a non-meagre set (such as \eqref{i:Non-meagre}) a quantitative condition in some ball. If the quantitative condition is stable under translations and weak$^*$ convergence, then Lemma \ref{l:GKL23} allows one to assume that the ball is centred at the origin, thereby letting one unleash the scaling symmetries of the system and prove a uniform estimate.


With a bit of work (see Corollary \ref{c:Navier-Stokes}) Theorem \ref{t:Relaxed} implies that Leray-Hopf solutions generically fail to belong to $\cup_{T>0} L^r(0,T;L^s)$ if $\frac{5}{2r} + \frac{3}{s} < 1$. Furthermore, uniform-in-$\nu$ bounds on $\norm{\u^\nu}_{L^r(0,\infty;L^s)}$ require that \eqref{e:Euler scaling} holds. We next consider the consequences of Theorem \ref{t:Relaxed} on Bochner-Besov spaces that satisfy \eqref{e:Euler scaling}.

\subsection{Weak solutions with Besov space regularity}
Recall that the regularity classes $L^r(0,\infty;\dot{B}_{r,\infty}^{\beta_r})$, $1 \leq r \leq 3$, $\beta_r$ as in \eqref{e:Besov exponent}, ensure the energy equality for weak solutions and satisfy the compatibility condition \eqref{e:Euler scaling}. By Theorem \ref{t:Relaxed}, $\beta_r$ are the only exponents compatible with uniform-in-viscosity bounds for 2-homogeneous spaces of initial data. The function space $L^3(0,T;B_{3,\infty}^{1/3})$ and its variant $L^3(0,T;B_{3,\infty}^{1/3}(\T^3))$ are natural in the study of high Reynolds number turbulence; see, for instance,~\cite{DE19,Eyi24,Shv10} and the references contained therein.

The function class $L^3(0,T;B_{3,\infty}^{1/3}(\T^3))$ is at a threshold for energy conservation in the Euler equations. Indeed, after several brilliant works using Nash's method of convex integration, starting with the seminal contributions of De Lellis and Sz\'ekelyhidi~\cite{DLS09,DLS13}, non-conservative $L^\infty(0,T;C^{1/3-}(\T^3))$ solutions were constructed by Isett~\cite{Ise18} and dissipative ones by Buckmaster, De Lellis, Sz\'ekelyhidi and Vicol in~\cite{BDSV19}; this gives the flexible side of the celebrated Onsager conjecture. Novack and Vicol constructed non-conservative solutions in $C([0,T];H^{1/2-}(\T^3) \cap L^{\infty-}(\T^3)) \hookrightarrow C([0,T];B^{1/3-}_{3,\infty}(\T^3))$~\cite{NV23}, and Giri, Kwon and Novack constructed non-conservative $C([0,T];B_{3,\infty}^{1/3-}(\T^3) \cap L^{\infty-}(\T^3))$ solutions with the local energy inequality~\cite{GKN23}. In the rigid side, Constantin, E and Titi showed that if $\beta > 1/3$, then $L^3(0,T;B_{3,\infty}^\beta(\T^3)) \cap C_w([0,T];L^2(\T^3))$ solutions of the Euler equations conserve kinetic energy~\cite{CET94}. In~\cite[Theorem 3.3]{CCFS08}, Cheskidov, Constantin, Friedlander and Shvydkoy extended conservation to the critical class $L^3(0,T;B_{3,c(\N)}^{1/3}) \cap C_w([0,T];L^2)$, where $B_{3,p}^{1/3} \subsetneq B_{3,c(\N)}^{1/3} \subsetneq B_{3,\infty}^{1/3}$ for all $p \in [1,\infty)$.

In the critical case, Cheskidov, Constantin, Friedlander and Shvydkoy also constructed a vector field $\vv \in B_{3,\infty}^{1/3}$ with a positive energy flux~\cite[\textsection 3.4]{CCFS08}, even though $\vv$ is not a solution of the Euler equations. It is strongly expected that some weak solutions of the Euler equations arise at the inviscid limit and dissipate kinetic energy, and $L^\infty(0,T;L^2_\sigma) \cap
L^3(0,T;\dot{B}_{3,\infty}^{1/3})$ appears to be a natural candidate space for such weak solutions; see also~\cite[Problem 4]{BV19A}. We also mention that Cheskidov and Shvydkoy have proved ill-posedness in $B_{3,\infty}^{1/3}(\T^3)$ in the following sense: there exists $\u_0 \in B_{3,\infty}^{1/3}(\T^3)$ such that if $\u \in C_w([0,T);L^2_\sigma)$ solves the Euler equations with $\u(\cdot,0) = \u_0$, then $\limsup_{t \to 0} \norm{\u(t)-\u_0}_{B_{3,\infty}^{1/3}} > 0$~\cite{CS10}. We refer to the reviews~\cite{BV19A,BV21,DLS19,DLS22,Eyi24,Shv10} for more on Onsager's turbulence theory and related topics.

It seems plausible that $\{\u_0 \in L^2_\sigma: \exists \u \in \NN_\nu(\u_0) \cap L^3(0,\infty;\dot{B}_{3,\infty}^{1/3})\}$ coincides with $L^2_\sigma \cap Z$ for some $2$-homogeneous Banach space of the form $Z = \dot{B}_{p,q}^{3(2-p)/(2p)}$ (perhaps $Z = \dot{B}_{2,2}^0 = L^2$). For definiteness, we extract the following corollary of Theorem \ref{t:Relaxed} (where one can also replace $L^3(0,\infty;\dot{B}_{3,\infty}^{1/3})$ by $L^r(0,\infty;\dot{B}_{r,\infty}^{\beta_r})$, $\frac{5}{3} \leq r < 3$).

\begin{corollary} \label{c:Relaxed}
Suppose $Z$ has the same properties as in Theorem \ref{t:Relaxed}. The following claims are equivalent:
\begin{enumerate}[(i)]
\item For some $\nu > 0$ and a non-meagre set of data $\u_0 \in L^2_\sigma \cap Z$ there exists a solution $\u \in \mathcal{N}_\nu(\u_0) \cap L^3(0,\infty;\dot{B}_{3,\infty}^{1/3})$.
\label{i:Besov existence in L2sigma}

\item There exists $C>0$ such that for every $\nu > 0$ and $\u_0 \in L^2_\sigma \cap Z$ there exists $\u^\nu \in \NN_\nu(\u_0)$ with
\[\norm{\u^\nu}_{L^3(0,\infty;\dot{B}_{3,\infty}^{1/3})}^3 \leq C \max\{\norm{\u_0}_{L^2}^2,\norm{\u_0}_Z^2\}.\]
Furthermore, for every $\u_0 \in L^2_\sigma \cap Z$ there exist $\nu_j \to 0$ such that $\u^{\nu_j}$ specified above converge weakly$^*$ in $L^\infty(0,\infty;L^2_\sigma)$ to a weak solution $\u$ of the Euler equations with $\norm{\u}_{L^3(0,\infty;\dot{B}_{3,\infty}^{1/3})}^3 \leq C \max\{\norm{\u_0}_{L^2}^2,\norm{\u_0}_Z^2\}$.
\label{i:Besov a priori estimate}
\end{enumerate}
If $Z$ is furthermore a closed subspace of $L^2_\sigma$, then \eqref{i:Besov existence in L2sigma}--\eqref{i:Besov a priori estimate} are equivalent to the following claim:
\begin{enumerate}[(i)]
\setcounter{enumi}{2}
\item For some $\nu > 0$ and a non-meagre set of data $\u_0 \in Z$ there exists a solution $\u \in \mathcal{N}_\nu(\u_0) \cap \cup_{T>0} L^3(0,T;\dot{B}_{3,\infty}^{1/3})$.
\label{i:Besov local existence}
\end{enumerate}
\end{corollary}

Note that claim \eqref{i:Besov a priori estimate} could, \emph{a priori}, be disproved by considering the Euler equations directly. A natural intermediate problem would be to determine the set $\{\u_0 \in L^2_\sigma: \exists \u \in L^\infty(0,\infty;L^2_\sigma) \cap L^3(0,\infty;\dot{B}_{3,\infty}^{1/3}) \text{ solving \eqref{e:NS1}--\eqref{e:NS2} with } \nu=0\}$, and by interpolation (see Lemma \ref{l:Interpolation}) one could consider solvability in $L^\infty(0,\infty;L^2_\sigma) \cap L^r(0,\infty;L^s)$ for $s \in (2,\frac{9}{2})$ and $\frac{5}{3r}+\frac{2}{s}=1$ instead. Similar problems also seem natural on $\T^3$, in bounded domains, in 2D and in other models such as MHD or SQG. Note that $\dot{B}_{3,\infty}^{1/3}(\R^2)$ is $6$-homogeneous, and as such, $L^3(0,\infty;\dot{B}_{3,\infty}^{1/3}(\R^2))$ is again an Onsager critical space (see~\cite{Shv10}) and satisfies the compatibility condition $2/r+2/s=1$ corresponding to \eqref{e:Euler scaling}; see Theorem \ref{t:Euler}.



\subsection{Time-weighted LPS classes and necessary and sufficient conditions for global regularity} \label{ss:Time-weighted LPS classes and necessary and sufficient conditions for global regularity}
A potential way to gain access to LPS classes in intervals $(0,T)$ is to introduce an increasing time weight; we denote
\[\norm{\vv}_{L^r_g(0,T;L^s)} \defeq \norm{\norm{\vv(t)}_{L^s}}_{L^r((0,T),g(t) \, dt)}\]
whenever $r,s \in [1,\infty]$ and $g \in L^1_{\tp{loc}}(0,\infty)$. Given LPS exponents $r,s$ with $s \in [3,\infty)$ it remains open whether all Leray-Hopf solutions belong to some class $L^r_g(0,\infty;L^s)$. We record several claims that are equivalent to a positive answer; each of them is a sufficient condition for global regularity for Schwartz data~\cite[Theorem 1.20]{Tao13}. 

\begin{theorem} \label{t:Main}
Let $\frac{2}{r}+\frac{3}{s} = 1$ with $s \in [3,\infty)$. The following claims are equivalent:

\begin{enumerate}[(i)]
\item \label{i:Time weight} There exist $\nu > 0$ and an increasing function $g \colon (0,\infty) \to (0,\infty)$ such that $\u^\nu \in \NN_\nu(\u_0) \cap \cup_{T > 0} L^r_g(0,T;L^s)$ exists for a non-meagre set of data $\u_0 \in L^2_\sigma$.

\item \label{i:L2 data} For every $\nu > 0$ and $\u_0 \in L^2_\sigma$ there exists $\u^\nu \in \mathcal{N}_\nu(\u_0) \cap \cap_{\tau > 0} L^r(\tau,\infty;L^s)$.

\item 
\label{i:A priori estimate for L2 data} There exists $G = G_{r,s} \in C[0,\infty)$ with $G(0) = 0$ and the following property: for every $\nu > 0$ and $\u_0 \in L^2_\sigma$ there exists $\u^\nu \in \mathcal{N}_\nu(\u_0)$ with
\begin{equation} \label{e:Inequality for G}
\norm{\u^\nu}_{L^r(\tau,\infty;L^s)} \leq \nu^{1 - \frac{1}{r}} G \left( \nu^{-\frac{5}{4}} \frac{\norm{\u_0}_{L^2}}{\tau^{\frac{1}{4}}} \right) \qquad \text{for all } \tau > 0.
\end{equation}

\item \label{i:L2 cap L3 data} For every $\nu > 0$ and $\u_0 \in L^2_\sigma \cap L^3$ there exists $\u^\nu \in \mathcal{N}_\nu(\u_0) \cap L^r(0,\infty;L^s)$.

\item \label{i:H1 mild} For every $\nu > 0$, $\u_0 \in H^1_\sigma$ and $T > 0$ there exists an $H^1$ mild solution $\u^\nu \in L^\infty(0,T;H^1_\sigma) \cap L^2(0,T;H^2)$ of \eqref{e:NS1}--\eqref{e:NS2} with initial datum $\u_0$.
\end{enumerate}
\end{theorem}

Theorem \ref{t:Main} complements~\cite[Figure 1]{Tao13} which contains necessary and sufficient conditions for \eqref{i:H1 mild}. In fact, the proof of the direction \eqref{i:H1 mild} $\Rightarrow$ \eqref{i:A priori estimate for L2 data} leans heavily on Tao's \emph{a priori} estimate for $H^1_\sigma$ data (formulated in Theorem \ref{t:Tao} of this paper). Note that if \eqref{i:Time weight}--\eqref{i:H1 mild} hold, then the Leray-Hopf solution given in \eqref{i:L2 data} satisfies the energy equality and is smooth in $\R^3 \times (0,\infty)$.


It is unclear to the author whether $\cap_{\tau > 0} L^r(\tau,\infty;L^s)$ integrability also implies uniqueness of Leray-Hopf solutions; it seems difficult to rule out instantaneous loss of uniqueness at $t = 0$. 
Motivated by this problem, in Theorem \ref{t:Dense to generic} we prove that condition \eqref{i:L2 cap L3 data} of Theorem \ref{t:Main} implies uniqueness for a \emph{Baire generic} $L^2_\sigma$ datum. 



\begin{theorem} \label{t:Dense to generic}
Let $\nu_0 > 0$, and suppose $Z \hookrightarrow L^2_\sigma$ is a Banach space. Each of the following claims holds for a $G_\delta$ set of data $\u_0 \in Z$.
\begin{enumerate}[(i)]
\item Every $\u \in \mathcal{N}_{\nu_0}(\u_0)$ satisfies the energy equality at a.e. $t > 0$. \label{i:Energy equality a.e.}

\item $\u \in \mathcal{N}_{\nu_0}(\u_0)$ is unique. \label{i:Uniqueness}

\item $\liminf_{\nu \to 0} \sup_{\u^\nu \in \NN_\nu(\u_0)} \int_0^T (\norm{\u_0}_{L^2}^2 - \norm{\u^\nu(t)}_{L^2}^2) \, dt = 0$ for every $T > 0$. (As a consequence,
\[\liminf_{\nu \to 0} \sup_{\u^\nu \in \NN_\nu(\u_0)} \nu \int_0^t \norm{\nabla \u^\nu(\tau)}_{L^2}^2 d\tau = 0\]
for every $t > 0$ and there exists a solution $\u \in C_w([0,\infty);L^2_\sigma)$ of the Euler equations with datum $\u_0$ such that $\norm{\u(t)}_{L^2} = \norm{\u_0}_{L^2}$ a.e. $t > 0$ and
\[\liminf_{\nu \to 0} \sup_{\u^\nu \in \NN_\nu(\u_0)} \norm{\u^\nu-\u}_{L^2(0,T;L^2)} = 0 \quad \text{for all } T > 0.)\]
\label{i:No dissipation anomaly}

\end{enumerate}
\end{theorem}
As a consequence of Theorem \ref{t:Dense to generic}, if global regularity holds in 3D Navier-Stokes equations, then uniqueness and a.e. energy equality must also hold for a Baire generic $L^2_\sigma$ datum (see Corollary \ref{c:Millennium (A) corollary}). In Corollary \ref{c:Millennium (C) corollary} we present a variant under external forcing. Global regularity in 3D Euler equations, in turn, would imply that \emph{anomalous energy dissipation} ($\liminf_{\nu \to 0} \nu \int_0^T \norm{\nabla \u^\nu(t)}_{L^2}^2 dt > 0$) fails for a Baire generic $L^2_\sigma$ datum (see Corollary \ref{c:Euler}). Theorem \ref{t:Dense to generic} also holds on the torus $\T^3$.

In \textsection \ref{s:Preliminaries} we recall background material on homogeneous Sobolev and Besov spaces and the Navier-Stokes equations. The proofs of Theorems \ref{t:Relaxed}--\ref{t:Dense to generic} and Corollary \ref{c:Relaxed} are presented in \textsection \ref{s:The proof of generic non-solvability theorem}--\ref{s:The proof of dense to generic theorem}, and applications of Theorem \ref{t:Dense to generic} to the global regularity problem are presented in \textsection \ref{s:Necessary conditions for global regularity}.

\section{Preliminaries} \label{s:Preliminaries}

\subsection{Homogeneous Sobolev and Besov spaces}
We very briefly recall relevant facts about homogeneous Sobolev and Besov spaces from~\cite{BCD11}. When $d \in \N$ and $s \in \R$, the homogeneous Sobolev space $\dot{H}^s(\R^d)$ consists of tempered distributions whose Fourier transform belongs to $L^1_{\tp{loc}}(\R^d)$ and satisfies
\[\norm{\vv}_{\dot{H}^s}^2 \defeq \int_{\R^d} \abs{\xi}^{2s} \abs{\hat{\vv}(\xi)}^2 d\xi < \infty.\]
Recall that $\dot{H}^s(\R^d)$ is a Hilbert space if and only if $s < \frac{d}{2}$~\cite[Proposition 1.34]{BCD11}, $\dot{H}^{1/2}(\R^3)$ embeds into $L^3(\R^3)$ (see~\cite[Theorem 1.38]{BCD11}) and $\dot{H}^{3/2}(\R^3) \cap L^1_{\tp{loc}}(\R^3)$ is contained in $\tp{BMO}(\R^3)$ with $\norm{\vv}_{\tp{BMO}} \leq C \norm{\vv}_{\dot{H}^{3/2}}$ (see~\cite[Theorem 1.48]{BCD11}).

We then define Besov spaces. We denote $A(r_1,r_2) \defeq \{x \in \R^d: r_1 < \abs{x} < r_2\}$ for $0 < r_1 < r_2$. Given $\alpha \in (1,\frac{4}{3})$, choose a radial function $\theta \in C_c^\infty(A(\frac{3}{4},\frac{8}{3}))$ with $\theta|_{A(1/\alpha,2\alpha)} = 1$. Now
\[\varphi \defeq \frac{\theta}{\sum_{j \in Z} \theta(2^{-j} \cdot))} \in C_c^\infty(A(3/4,8/3)), \quad \chi \defeq 1 - \sum_{j=0}^\infty \varphi(2^{-j} \cdot) \in C_c^\infty(B(0,4/3))\]
are well-defined (with locally finite sums)~\cite[p. 60]{BCD11}; see~\cite[Proposition 2.10]{BCD11} for further properties of $\varphi$ and $\chi$. We denote
\[\dot{\Delta}_j g \defeq \varphi(2^{-j} D) g \defeq \mathcal{F}^{-1}(\varphi(2^{-j} \cdot)) * g\]
for all $g \in \mathcal{S}'(\R^d)$ and $j \in \Z$, where $\mathcal{F}$ is the Fourier transform. We also denote $\dot{S}_j g \defeq \chi(2^{-j} D) g$.

Let $d \in \N$, $\beta \in \R$ and $p,q \in [1,\infty]$, and fix $\vv \in \mathcal{S}'(\R^d)$. We define the homogeneous Besov seminorm
\[\norm{\vv}_{\dot{B}_{p,q}^\beta} \defeq \norm{ \left( 2^{j\beta} \norm{\dot{\Delta}_j \vv}_{L^p} \right)_{j \in \Z}}_{\ell^q(\Z)}.\]
and the inhomogeneous Besov norm
\[\norm{\vv}_{B_{p,q}^\beta} \defeq \norm{ \left( 2^{j\beta} \norm{\dot{\Delta}_j \vv}_{L^p} \right)_{j \in \N}}_{\ell^q(\N)} + \norm{\dot{S}_0 \vv}_{L^p}.\]
Following~\cite{BCD11} we denote by $\S_h'(\R^d)$ the set of tempered distributions $\vv$ such that $\lim_{\lambda \to \infty} \norm{\theta(\lambda D) \vv}_{L^\infty} = 0$ for all $\theta \in \mathcal{D}(\R^d)$. In particular, $\S_h'(\R^d)$ contains all compactly supported distributions and excludes all non-zero polynomials~\cite[pp. 22-23]{BCD11}.

\begin{definition}
Let $d \in \N$, $\beta \in \R$ and $p,q \in [1,\infty]$. The homogeneous Besov space $\dot{B}_{p,q}^\beta(\R^d)$ consists of $\vv \in \S'_h(\R^d)$ such that $\norm{\vv}_{\dot{B}_{p,q}^\beta} < \infty$, and the inhomogeneous Besov space $B_{p,q}^\beta(\R^d)$ consists of $\vv \in \S'(\R^d)$ such that $\norm{\vv}_{B_{p,q}^\beta} < \infty$.

If $K \subset \R^d$ is compact, we denote by $B_{p,q}^\beta(K)$ and $\dot{B}_{p,q}^\beta(K)$ the members of $B_{p,q}^\beta(\R^d)$ and $\dot{B}_{p,q}^\beta(\R^d)$, respectively, supported on $K$.
\end{definition}

Homogeneous Besov spaces $\dot{B}_{p,r}^\beta(\R^d)$ are normed spaces~\cite[Proposition 2.16]{BCD11}, and they are complete if either $\beta < \frac{d}{p}$ or $r = 1$ and $\beta \leq \frac{d}{p}$~\cite[Theorem 2.25]{BCD11}. By using Young's inequality on $\|\dot{S}_0 \vv\|_{L^p}$ we get $\dot{B}_{p,q}^\beta(\R^d) \cap L^r(\R^d) \hookrightarrow B_{p,q}^\beta(\R^d)$ whenever $1 \leq r \leq p$.

We have
\begin{equation} \label{e:Homogeneity of Besov spaces}
\norm{\vv(\lambda^{-1} \cdot)}_{\dot{B}_{p,q}^\beta(\R^d)} \approx \lambda^{\frac{d}{p}-\beta} \norm{\vv}_{\dot{B}_{p,q}^\beta(\R^d)}
\end{equation}
for all $\vv \in \dot{B}_{p,q}^\beta(\R^d)$ and $\lambda > 0$~\cite[Remark 2.19]{BCD11}. 
We also have the embeddings
\begin{align}
& p_1 \leq p_2, \; q_1 \leq q_2 \quad \Longrightarrow \quad \dot{B}_{p_1,q_1}^\beta(\R^d) \hookrightarrow \dot{B}_{p_2,q_2}^{\beta-\frac{d}{p_1}+\frac{d}{p_2}}(\R^d), \label{e:BCD11  Proposition 2.20} \\
& p \leq q \quad \Longrightarrow \quad \dot{B}_{p,1}^{\frac{d}{p} - \frac{d}{q}} \hookrightarrow L^q \hookrightarrow \dot{B}_{q,\max\{q,2\}}^0 \label{e:Lp Besov embeddings}
\end{align}
\cite[Propositions 2.20 and 2.39, Theorems 2.40 and 2.41]{BCD11} and the interpolation result
\begin{equation} \label{e:Interpolation of Besov spaces}
\norm{\vv}_{\dot{B}_{p,1}^{\theta \beta_1 + (1-\theta) \beta_2}(\R^d)} \leq \frac{C}{\beta_2-\beta_1} \left( \frac{1}{\theta} + \frac{1}{1-\theta} \right) \norm{\vv}_{\dot{B}_{p,\infty}^{\beta_1}(\R^d)}^\theta \norm{\vv}_{\dot{B}_{p,\infty}^{\beta_2}(\R^d)}^{1-\theta}
\end{equation}
whenever $\beta_1 < \beta_2$ and $\theta \in (0,1)$~\cite[Proposition 2.22]{BCD11}. A short computation gives $\norm{\vv(\cdot-x_0)}_{\dot{B}_{p,q}^\beta(\R^d)} = \norm{\vv}_{\dot{B}_{p,q}^\beta(\R^d)}$ for all $\vv \in \dot{B}_{p,q}^\beta(\R^d)$ and $x_0 \in \R^d$.

Suppose now $1 \leq p \leq \infty$ and $-\frac{d}{p'} \leq \beta < \frac{d}{p}$. If $\varphi \in C_c^\infty(\R^d)$ and $K \supset \tp{supp}(\varphi)$ is compact, then $g \mapsto \varphi g \colon \dot{B}_{p,\infty}^\beta(\R^d) \to B_{p,\infty}^\beta(K)$ is a continuous embedding~\cite[Corollary 2.1.1]{DM15}. Whenever $-\infty < \beta' < \beta < \infty$ and $K \subset \R^d$ is compact, $B_{p,\infty}^\beta(K)$ embeds compactly into $B_{p,1}^{\beta'}(K)$~\cite[Corollary 2.96]{BCD11}. Combining with \eqref{e:Lp Besov embeddings} and multiplying by $\varphi \in C_c^\infty(\R^d)$, $\varphi|_K = 1$, we have  the following (recall that $\beta_p \defeq \frac{11-3p}{2p}$):

\begin{corollary} \label{c:Besov}
Let $d = 3$ and $p \in [\frac{5}{3},\frac{11}{3})$, and suppose $K \subset \R^3$ is compact. Then $B_{p,\infty}^{\beta_p}(K)$ embeds compactly into $L^q(K)$ for all $q \in [1,\frac{6p}{3p-5})$.
\end{corollary}

For the proof of Corollary \ref{c:Besov} we will need the following:

\begin{lemma} \label{l:Fatou}
Whenever $p,r,s \in (1,\infty]$ and $\beta \in \R$, the Bochner-Besov space $L^p(0,\infty;\dot{B}_{p,\infty}^\beta(\R^d))$ has the Fatou property with respect to $L^r(0,\infty;L^s(\R^d))$.
\end{lemma}

\begin{proof}
Suppose $\vv_k \overset{*}{\rightharpoonup} \vv$ in $L^r(0,\infty;L^s(\R^d))$ and $\sup_{k \in \N} \norm{\vv_k}_{L^p(0,\infty;\dot{B}_{p,q}^\beta(\R^d))} < \infty$. We first show that
\begin{equation} \label{e:Inequality for dyadic blocks}
\int_0^\infty \norm{\dot{\Delta}_j \vv(t)}_{L^p(\R^d)}^p dt \leq \liminf_{k \to \infty} \int_0^\infty \norm{\dot{\Delta}_j \vv_k(t)}_{L^p(\R^d)}^p dt
\end{equation}
for all $j \in \Z$. Let $j \in \Z$ and $\epsilon > 0$. Now there exists $\psi_j \in C_c^\infty(\R^d \times (0,\infty))$ such that $\norm{\psi_j}_{L^{p'}(0,\infty;L^{p'}(\R^d))} \leq 1$ and
\[\int_0^\infty \int_{\R^d} \vv(x,t) \cdot \psi_j(x,t) \, dx \, dt \geq (1-\epsilon) \norm{\dot{\Delta}_j \vv}_{L^p(0,\infty;L^p(\R^d))}.\]
By using the assumption that $\vv_k \overset{*}{\rightharpoonup} \vv$ in $L^r(0,\infty;L^s(\R^d))$ we get
\begin{align*}
(1-\epsilon)^p \int_0^\infty \norm{\dot{\Delta}_j \vv(t)}_{L^p(\R^d)}^p dt
&\leq \lim_{k \to \infty} \left( \int_0^\infty \int_{\R^d} \dot{\Delta}_j \vv_k(x,t) \cdot \psi_j(x,t) \, dx \, dt \right)^p \\
&\leq \liminf_{k \to \infty} \int_0^\infty \norm{\dot{\Delta}_j \vv_k(t)}_{L^p(\R^d)}^p dt,
\end{align*}
so that \eqref{e:Inequality for dyadic blocks} holds.

We then finish the proof of the Lemma via Fatou's lemma and \eqref{e:Inequality for dyadic blocks}:
\begin{align*}
\norm{\vv}_{L^p(0,\infty;\dot{B}_{p,\infty}^\beta)}^p
&= \int_0^\infty \lim_{\ell \to \infty} \sup_{\abs{j} \leq \ell} 2^{j\beta p} \norm{\dot{\Delta}_j \vv(t)}_{L^p}^p dt \\
&\leq \liminf_{\ell \to \infty} \liminf_{k \to \infty} \int_0^T \sup_{\abs{j} \leq \ell} 2^{j\beta p} \norm{\dot{\Delta}_j \vv_k(t)}_{L^p}^p dt \\
&\leq \liminf_{\ell \to \infty} \liminf_{k \to \infty} \int_0^T \sup_{j \in \Z} 2^{j\beta p} \norm{\dot{\Delta}_j \vv_k(t)}_{L^p}^p dt \\
&= \liminf_{k \to \infty} \norm{\vv_k}_{L^p(0,\infty;\dot{B}_{p,\infty}^\beta)}^p,
\end{align*}
as claimed.
\end{proof}

For all $p \in [1,\infty]$ and $\beta \in (0,1)$, the Besov space $\dot{B}_{p,\infty}^\beta(\R^d)$ has the equivalent norm $\sup_{h \neq 0} \abs{h}^{-\beta} \norm{\vv(\cdot+h)-\vv}_{L^p(\R^d)}$~\cite[Theorem 2.36]{BCD11}. As such, $\dot{B}_{p,\infty}^\beta$ measures, loosely speaking, H\"older continuity on average. On the torus, it is customary to use the norm
\[\norm{\vv}_{B_{p,\infty}^\beta(\T^3)} \defeq \norm{\vv}_{L^p(\T^3)} + \sup_{h \neq 0} \frac{\norm{\vv(\cdot+h)-\vv}_{L^p(\T^3)}}{\abs{h}^\beta}\]
for this range of exponents.

Homogeneous Besov spaces with negative smoothness $\beta$ can also be characterised in terms of the heat kernel. In particular, if $r,s \in [1,\infty]$, then
\begin{equation} \label{e:Heat characterisation of Besov spaces}
\norm{\vv}_{\dot{B}_{s,r}^{-2/r}(\R^d)} \approx \norm{e^{t \Delta} \vv}_{L^r(0,\infty;L^s(\R^d))}
\end{equation}
for all $\vv \in \S'_h(\R^d)$~\cite[Theorem 2.34]{BCD11}. We single out the following consequence (which of course also has a more direct proof).
\begin{proposition} \label{c:Incompatible Lp}
If $r,s,p \in [1,\infty]$ satisfy $\frac{2}{r}+\frac{d}{s} < \frac{d}{p}$ and $Z \subset \S'_h(\R^d)$ is a $p$-homogeneous Banach space, then we have $e^{t\Delta} \vv \in L^r(0,\infty;L^s(\R^d))$ only for a meagre set of vector fields $\vv \in Z$.
\end{proposition}

\begin{proof}
If Proposition \ref{c:Incompatible Lp} failed, the Open mapping theorem would imply that $\norm{\vv}_{L^p} \approx \norm{\vv}_{L^p} + \norm{\vv}_{\dot{B}_{s,r}^{-2/r}}$ for all $\vv \in Z$, leading to a contradiction with \eqref{e:Homogeneity of Besov spaces}.
\end{proof}

In the proof of Corollary \ref{c:Relaxed} we will need the integrability properties of those Leray-Hopf solutions which belong to $L^r(0,\infty;\dot{B}_{r,\infty}^{\beta_r})$ (see \eqref{e:Besov exponent}).
\begin{lemma} \label{l:Interpolation}
If $d=3$, $2 \leq r < \frac{11}{3}$, $\theta \in (0,1)$, $\theta \leq \frac{11-3r}{5}$ and $T \in \R_+ \cup \{\infty\}$, then
\[L^\infty(0,T;L^2) \cap L^r(0,T;\dot{B}_{r,\infty}^{\beta_r}) \hookrightarrow L^{\frac{r}{1-\theta}}(0,T;L^{\frac{6r}{3r-5(1-\theta)}}).\]
In particular,
\[L^\infty(0,T;L^2) \cap L^r(0,T;\dot{B}_{r,\infty}^{\beta_r}) \hookrightarrow L^r(0,T;B_{r,\infty}^{\beta_r})\]
for all $r \in [2,\frac{11}{3})$ and $T \in \R_+$.
\end{lemma}

\begin{proof}
Let $\vv \in L^\infty(0,T;L^2) \cap L^r(0,T;\dot{B}_{r,\infty}^{\beta_r})$ and $\theta \in (0,1)$. By \eqref{e:BCD11  Proposition 2.20}--\eqref{e:Lp Besov embeddings}, $L^2 \hookrightarrow \dot{B}_{2,2}^0 \hookrightarrow \dot{B}_{r,\infty}^{3(2-r)/(2r)}$, and so \eqref{e:BCD11  Proposition 2.20}--\eqref{e:Interpolation of Besov spaces} yield
\begin{equation} \label{e:Besov interpolation}
\norm{\vv(t)}_{L^{\frac{6r}{3r-5(1-\theta)}}} \lesssim_\theta \norm{\vv(t)}_{\dot{B}_{r,1}^{\frac{3(2-r)}{2r} \theta + \beta_r (1-\theta)}} \lesssim_\theta \norm{\vv(t)}_{L^2}^\theta \norm{\vv(t)}_{\dot{B}_{r,\infty}^{\beta_r}}^{1-\theta}
\end{equation}
for a.e. $t \in (0,T)$; here we use the assumption $\theta \leq \frac{11-3r}{5}$ to get $\frac{3(2-r)\theta}{2r} + \beta_r(1-\theta) \geq 0$, so that \eqref{e:Lp Besov embeddings} is applicable. Now \eqref{e:Besov interpolation} yields
\[\norm{\vv}_{L^{\frac{r}{1-\theta}}(0,T;L^{\frac{6r}{3r-5(1-\theta)}})} \lesssim_\theta \norm{\vv}_{L^\infty(0,\infty;L^2)}^\theta \norm{\vv}_{L^r(0,T;\dot{B}_{r,\infty}^{\beta_r})}^{1-\theta}.\]
The second embedding follows from the fact that $L^r \cap \dot{B}_{r,\infty}^{\beta_r} \hookrightarrow B_{r,\infty}^{\beta_r}$: if $r > 2$, we choose $\theta = \frac{11-3r}{5} \in (0,1)$, and if $r = 2$, we use the fact that $L^\infty(0,T;L^2) \hookrightarrow L^2(0,T;L^2)$.
\end{proof}

\begin{remark} \label{r:Besov integrability remark}
By \eqref{e:Besov interpolation}, $L^2 \cap \dot{B}_{r,\infty}^{\beta_r} \hookrightarrow L^s$ for all $s \in [2,\frac{6r}{3r-5})$. In the converse direction we have, e.g., $\dot{H}^{5/(2r)} = \dot{B}_{2,2}^{5/(2r)} \hookrightarrow \dot{B}_{r,\infty}^{\beta_r}$~\cite[p. 63]{BCD11}. Therefore, every Leray-Hopf solution belongs to $\cup_{T > 0} L^r(T,\infty;\dot{B}_{r,\infty}^{\beta_r})$; this follows e.g. by using~\cite[Theorem B]{BZZ19}.
\end{remark}

\subsection{Navier-Stokes equations}
We list various classical results on weak and Leray-Hopf solutions that are used later. For a proof of the first one see e.g.~\cite[Theorems 8.14 and 14.4]{RRS16}.

\begin{theorem} \label{t:Leray}
For every $\u_0 \in L^2_\sigma$ there exists a Leray-Hopf solution $\u$ which satisfies, moreover, the \emph{strong energy inequality}
\[\frac{1}{2} \norm{\u(t)}_{L^2}^2 + \nu \int_{t'}^t \norm{\nabla \u(\tau)}_{L^2}^2 d\tau \leq \frac{1}{2} \norm{\u(t')}_{L^2}^2\]
for every $t > t'$ and a.e. $t' \geq 0$, including $t' = 0$. Furthermore, every such solution is smooth on an open set of times of full measure.
\end{theorem}

For a proof of the next result see e.g.~\cite[Theorem 7.1]{Gal00}:
\begin{theorem} \label{t:L2 cap L3 data local solvability}
Suppose $\frac{2}{r}+\frac{3}{s}=1$ with $s \in [3,\infty)$. If $\u_0 \in L^2_\sigma \cap L^3$, then there exist $T > 0$ and a unique weak solution $\u \in L^\infty(0,T;L^2_\sigma) \cap L^2(0,T;\dot{H}^1) \cap L^r(0,T;L^s)$.
\end{theorem}

 We then recall a version of \emph{weak-strong uniqueness} which says that a Leray-Hopf solution is unique if it belongs to an LPS class, see e.g.~\cite[Theorems 4.2 and 7.2]{Gal00}:

\begin{theorem} \label{t:Weak-strong uniqueness}
Let $\frac{2}{r}+\frac{3}{s}=1$ with $s \in [3,\infty)$. Let $\u_0 \in L^2_\sigma$ and $\u \in \mathcal{N}(\u_0)$, and suppose there exists a weak solution $\vv \in L^\infty(0,T;L^2_\sigma) \cap L^2(0,T;\dot{H}^1) \cap L^r(0,T;L^s)$ with datum $\u_0$. Then $\u = \vv$ in $\R^3 \times [0,T)$.
\end{theorem}

Aside from smoothness, the LPS conditions imply higher integrability; see~\cite[Theorem 1.3]{ESS03} and~\cite[Theorems 7.3 and 8.17]{RRS16}.

\begin{theorem} \label{t:Higher integrability}
Let $\frac{2}{r}+\frac{3}{s}=1$ with $s \in [3,\infty)$ and $\u_0 \in L^2_\sigma$. If $\u \in \mathcal{N}(\u_0) \cap L^r(0,T;L^s)$, then $\u \in \cap_{m \in \N} \cap_{0 < \epsilon < T} C([\epsilon,T];H^m)$.
\end{theorem}

We will also need the following result of Galdi~\cite[Theorem 1.1]{Gal19} in the proof of Theorem \ref{t:Main}.
\begin{theorem} \label{t:Galdi19}
If $\frac{2}{r}+\frac{3}{s}=1$ with $s \in [3,\infty)$ and
\[\u \in \begin{cases} \cap_{\tau > 0} L^r(\tau,\infty;L^s), & s > 3, \\
\cap_{\tau > 0} C([\tau,\infty);L^3), & s = 3
\end{cases}\]
is a very weak solution with $\u(t) \rightharpoonup \u(0)$ as $t \to 0$, then $\u$ is a Leray-Hopf solution.
\end{theorem}

The proof of Theorem \ref{t:Main} also uses heavily an \emph{a priori} estimate of Tao on mild $H^1$ solutions. For the definition of such solutions recall that the \emph{Leray Projection} $\mathbb{P} \defeq [\delta_{ij} + R_i R_j]_{i,j=1}^3$ (where $R_i$ are Riesz transforms) maps $L^p$ boundedly onto $L^p_\sigma$ for all $p \in (1,\infty)$.

\begin{definition}
Suppose $\u_0 \in H^1_\sigma$. Then $\u \in L^\infty(0,T;H^1_\sigma) \cap L^2(0,T;H^2)$ is called an \emph{$H^1$ mild solution} of \eqref{e:NS1}--\eqref{e:NS2} if
\[\u(t) = e^{\nu t\Delta} \u_0 - \int_0^t e^{\nu (t-\tau) \Delta} \mathbb{P} [\u(\tau) \cdot \nabla \u(\tau)] \, d\tau\]
for a.e. $t \in [0,T]$.
\end{definition}

Given $\u_0 \in H^1_\sigma$, a mapping $\u \in L^\infty(0,T;H^1_\sigma) \cap L^2(0,T;H^2)$ is an $H^1$ mild solution with datum $\u_0$ if and only if $\u$ is a weak solution in $\R^3 \times [0,T)$, i.e., satisfies \eqref{e:Weak solution}; see~\cite[Theorem 2.1]{FJR72}. We recall a classical local existence and uniqueness result, see e.g.~\cite[Theorem 5.4]{Tao13}:

\begin{theorem} \label{t:Tao 2}
There exist $c,C > 0$ with the following property: if $\norm{\u_0}_{H^1}^4 T \leq c$, then there exists a unique $H^1$ mild solution $\u \in L^\infty(0,T;H^1) \cap L^2(0,T;H^2)$, and furthermore $\norm{\u}_{L^\infty(0,T;H^1)} + \norm{\u}_{L^2(0,T;H^2)} \leq C \norm{\u_0}_{H^1}$.
\end{theorem}

Global existence of $H^1$ mild solutions remains open. We recall an equivalent condition, an \emph{a priori} estimate given in ~\cite[Theorem 1.20(vi)]{Tao13}.

\begin{theorem} \label{t:Tao}
Suppose that for every $\u_0 \in H^1_\sigma$ and $T > 0$ there exists an $H^1$ mild solution $\u \in L^\infty(0,T;H^1_\sigma) \cap L^2(0,T;H^2)$. Then there exists a function $F \colon (0,\infty) \to (0,\infty)$ with the following property: whenever $\u \in L^\infty(0,T;H^1_\sigma) \cap L^2(0,T;H^2) \cap C^\infty(\R^3 \times [0,T])$ is a solution of \eqref{e:NS1}--\eqref{e:NS2} with $\norm{\u_0}_{H^1} \leq A < \infty$, we have
\[\norm{\u}_{L^\infty(0,T;H^1)} \leq F(A).\]
\end{theorem}

In Theorem \ref{t:Tao}, we can set $F(A) = A$ for small $A > 0$; see e.g.~\cite[pp. 141--142]{RRS16}. We can of course also assume that $F$ is increasing. There are also similar \emph{a priori} estimates of Gallagher~\cite[Corollary 1]{Gal00}, Rusin and \v{S}ver\'ak~\cite[Corollary 4.4]{RS11} as well as Jia and \v{S}ver\'ak~\cite[Corollary 4.2]{JS13} for data in critical spaces such as $\dot{H}^{1/2}$ and $L^3$. However, the subcritical result of Theorem \ref{t:Tao} suffices for the purposes of this paper.

We will also need a global well-posedness result for data with small $\dot{H}^{1/2}$ norm; see ~\cite[Theorem 5.6 and Proposition 5.16]{BCD11} as well as the proof of~\cite[Proposition 5.13]{BCD11}.

\begin{theorem} \label{t:Fujita-Kato}
There is a constant $c > 0$ such that if $\u_0 \in H^{1/2}$ with $\nabla \cdot \u_0 = 0$ and $\norm{\u_0}_{\dot{H}^{1/2}} \leq c \nu$, then there exists $\u \in \NN_\nu(\u_0) \cap C([0,\infty);\dot{H}^{1/2}) \cap L^2(0,\infty;\dot{H}^{3/2})$ that satisfies
\[\norm{\u(t)}_{\dot{H}^{1/2}}^2 + \nu \int_0^t \norm{\nabla \u(\tau)}_{\dot{H}^{1/2}}^2 \, d\tau \leq \norm{\u_0}_{\dot{H}^{1/2}}^2\]
for all $t \geq 0$.
\end{theorem}

\begin{remark} \label{r:Embeddings}
Since $\dot{H}^{1/2}$ embeds into $L^3$, Theorem \ref{t:Weak-strong uniqueness} implies that the solution $\u$ given by Theorem \ref{t:Fujita-Kato} is a unique Leray-Hopf solution with datum $\u_0$. By the inequality $\norm{\u(t)}_{\tp{BMO}} \leq C \norm{\u(t)}_{\dot{H}^{3/2}}$ we get $\norm{\u}_{L^2(0,\infty;\tp{BMO})} \leq C \norm{\u_0}_{\dot{H}^{1/2}}/\sqrt{\nu}$.
\end{remark}

To close the chapter, we discuss the weak$^*$ compactness properties of weak and Leray-Hopf solutions.
\begin{proposition} \label{p:Weak* stability}
If $\u_{0,j} \to \u_0$ in $L^2_\sigma$ and $\u_j \in \mathcal{N}_\nu(\u_{0,j})$ for every $j \in \N$, then for a subsequence, $\u_j \overset{*}{\rightharpoonup} \u \in \mathcal{N}_\nu(\u_0)$ in $L^\infty(0,\infty;L^2_\sigma) \cap L^2(0,\infty;\dot{H}^1)$.

If $\u_{0,j} \rightharpoonup \u_0$ in $L^2_\sigma$ and $\u_j \in \mathcal{W}_\nu(\u_{0,j})$ satisfy $\u_j \overset{*}{\rightharpoonup} \u$ in $L^\infty(0,\infty;L^2_\sigma) \cap L^2(0,\infty;\dot{H}^1)$, then $\u \in \mathcal{W}_\nu(\u_0)$.
\end{proposition}

\begin{remark} \label{r:Weak* stability}
Proposition \ref{p:Weak* stability} is proved by standard weak compactness methods; see e.g.~\cite[p. 103]{RRS16}. It is, however, open whether the class of Leray-Hopf solutions is stable under weak$^*$ convergence (i.e. whether $\u_{0,j} \rightharpoonup \u_0$ and $\u_j \in \mathcal{N}_\nu(\u_{0,j})$ imply $\u_j \overset{*}{\rightharpoonup} \u \in \mathcal{N}_\nu(\u_0)$ for a subsequence). Seregin has studied the problem for $\u_0 = 0$ in~\cite{Ser12} and shown that one sufficient condition is that each $\u_j$ satisfies the local energy inequality and $\sup_{j \in \N} \norm{\u_{0,j}}_{L^s} < \infty$ for some $s > 3$. Stability under weak$^*$ convergence has also been shown for all $L^3$ data and so-called weak $L^3$-solutions~\cite{SS17} and Leray solutions~\cite{RS11}; for other critical classes of initial data see e.g.~\cite{AB19,BSS18} and the references contained therein.
\end{remark}

\section{The case of Euler equations} \label{s:The proof of generic non-solvability theorem}
In Theorem \ref{t:Euler} we prove the generic lack of higher integrability of weak solutions of the Euler equations (in the $L^r(0,\infty;L^s)$ scale), following the proof of~\cite[Lemma 5.1]{GKL23}. Wiedemann~\cite{Wie11} has constructed a weak solution $\u \in C_w([0,\infty);L^2_\sigma(\T^d))$ for every $\u_0 \in L^2_\sigma(\T^d)$ and $d \geq 2$ by leaning on convex integration techniques from~\cite{DLS10}. The author is not aware of similar results for $L^p_\sigma$ data, $p > 2$. For the proof of Theorem \ref{t:Euler} we recall a linear relaxation of the Euler equations (essentially) from~\cite{DLS09}. Below, $S$ takes values in $\R^{d \times d}_{\tp{sym}}$.

\begin{definition}
Let $d \in \N$, $d \geq 2$. Suppose $\u_0 \in L^2_{\tp{loc}}(\R^d)$ with $\nabla \cdot \u = 0$. We say that $(\u,S) \in L^2_{\tp{loc}}(\R^d \times [0,T)) \times L^1_{\tp{loc}}(\R^d \times [0,T))$ is a weak solution of the \emph{relaxed Euler equations}
\begin{equation} \label{e:Relaxed NS}
\partial_t \u + \nabla \cdot S + \nabla p = 0, \quad \nabla \cdot \u = 0, \quad \u(\cdot,0) = \u_0
\end{equation}
if $\nabla \cdot \u(\cdot,t) = 0$ a.e. $t \in (0,T)$ and
\[\int_0^T \int_{\R^d} (\u \cdot \partial_t \varphi + S \colon \nabla \varphi) + \int_{\R^d} \u_0 \cdot \varphi(\cdot,0) = 0\]
for all $\varphi \in C_c^\infty([0,T) \times \R^d)$ with $\nabla \cdot \varphi = 0$.
\end{definition}

If $\u \in L^2_{\tp{loc}}(\R^d \times [0,T))$ is a weak solution of Euler equations with datum $\u_0 \in L^2_{\tp{loc}}(\R^d)$, $\nabla \cdot \u_0 = 0$, then $(\u,\u \otimes \u)$ is clearly a weak solution of the relaxed Euler equations with $\u(\cdot,0) = \u_0$.

\begin{theorem} \label{t:Euler}
Let $d \in \N$, $d \geq 2$, $p \in [2,\infty]$ and $r,s \in (2,\infty]$. Suppose $\{0\} \subsetneq Z \subset L^2_{\sigma,\tp{loc}}$ is a $p$-homogeneous Banach space. If
\[\frac{1+\frac{d}{p}}{r} + \frac{d}{s} < \frac{d}{p},\]
then a weak solution $\u \in \cup_{T>0} L^r(0,T;L^s(\R^d))$ of the Euler equations only exists for a meagre set of data $\u_0 \in Z$. If $\frac{1+d/p}{r}+\frac{d}{s} > \frac{d}{p}$, an analogous result holds for $L^r(0,\infty;L^s)$.
\end{theorem}

\begin{proof}
We prove the case $\frac{1+d/p}{r} + \frac{d}{s} < \frac{d}{p}$. Note that it suffices to show, for every fixed $T > 0$, that a weak solution $(\u,S) \in L^r(0,T;L^s) \times L^{r/2}(0,T;L^{s/2})$ of \eqref{e:Relaxed NS} only exists for a meagre set of data.

Let $T > 0$. Seeking a contradiction, suppose that for a non-meagre set of data, a solution $(\u,S) \in L^r(0,T;L^s) \times L^{r/2}(0,T;L^{s/2})$ exists. Thus, by definition, one of the (manifestly closed) sets
\[X_M \defeq \{\u_0 \in X \colon \exists (\u,S) \text{ solving \eqref{e:Relaxed NS} with} \norm{\u}_{L^r(0,T;L^s)} + \norm{S}_{L^{r/2}(0,T;L^{s/2})} \leq M\}\]
contains a closed ball $\bar{B}(\vv_0,\epsilon) \subset Z$. By linearity, we may assume that $\bar{\vv}_0 = 0$ (by possibly enlarging $M$). By linearity again, for some $C > 0$, for every $\u_0 \in Z$ there exists a solution $(\u,S)$ of \eqref{e:Relaxed NS} with
\[\norm{\u}_{L^r(0,T;L^s)} + \norm{S}_{L^{r/2}(0,T;L^{s/2})} \leq C \norm{\u_0}_Z.\]

Fix now $\u_0 \in Z \setminus \{0\}$ and define $\tilde{\u}_0 \defeq \u_0(\lambda \cdot)$ so that $\norm{\tilde{\u}_0}_Z = \lambda^{-d/p} \norm{\u_0}_Z$. Select a solution $(\u,S)$ of the relaxed Euler equations with $\u(\cdot,0) = \tilde{\u}_0$ such that
\[\norm{\u}_{L^r(0,T;L^s)} + \norm{S}_{L^{r/2}(0,T;L^{s/2})}
\leq C \norm{\tilde{\u}_0}_Z = C \lambda^{-d/p} \norm{\u_0}_Z.\]
Let $\gamma > 0$, to be chosen later. Now
\[\u_\lambda(\x,t) \defeq \u \left( \frac{\x}{\lambda}, \frac{t}{\lambda^\gamma} \right), \qquad S_\lambda(\x,t) \defeq \lambda^{1-\gamma} S \left( \frac{\x}{\lambda}, \frac{t}{\lambda^\gamma} \right)\]
form a solution of \eqref{e:Relaxed NS} on $\R^d \times [0,\lambda^\gamma T)$ with $\u_\lambda(\cdot,0) = \u_0$. Furthermore,
\begin{align*}
& \norm{\u_\lambda}_{L^r(0,\lambda^\gamma T;L^s)} = \lambda^{\frac{\gamma}{r} + \frac{d}{s}} \norm{\u}_{L^r(0,T;L^s)} \leq C \lambda^{\frac{\gamma}{r} + \frac{d}{s} - \frac{d}{p}} \norm{\u_0}_Z, \\
& \norm{S_\lambda}_{L^{r/2}(0,\lambda^\gamma T;L^{s/2})} = \lambda^{\frac{2\gamma}{r} + \frac{2d}{s} + 1 - \gamma} \norm{S}_{L^{r/2}(0,T;L^{s/2})} \leq C \lambda^{\frac{2\gamma}{r} + \frac{2d}{s} + 1 - \gamma - \frac{d}{p}} \norm{\u_0}_Z.
\end{align*}
If $r < \infty$, then $\gamma \defeq r (\frac{d}{p}-\frac{d}{s}) > 0$ by assumption and 
\begin{equation} \label{e:Inequality on exponents}
\frac{2\gamma}{r} + \frac{2d}{s} + 1 - \gamma - \frac{d}{p} = \frac{(1-r)d}{p} + \frac{rd}{s} + 1 < 0.
\end{equation}
Therefore, making $\gamma$ slightly smaller, we get both \eqref{e:Inequality on exponents} and $\frac{\gamma}{r}+\frac{d}{s}-\frac{d}{p} < 0$. Letting $\lambda \to \infty$ we conclude that $(0,0)$ is a weak solution of \eqref{e:Relaxed NS} with datum $\u_0 \in Z \setminus \{0\}$. This gives the sought contradiction when $r < \infty$. If $r = \infty$, the same argument works for any $\gamma > \frac{2d}{s}+1-\frac{d}{p}$.

The second claim has an entirely analogous proof, but $(0,T)$ is replaced by $(0,\infty)$, $\gamma = r(\frac{d}{p}-\frac{d}{s})$ need not be positive and we let $\lambda \to 0$.
\end{proof}

\begin{remark}
A similar relaxation of the Navier-Stokes equations shows that if $\frac{5}{2r}+\frac{3}{s} < \frac{3}{2}$, then not only weak solutions but also very weak solutions in $\cup_{T>0} L^r(0,T;L^s)$ are ruled out for a Baire generic datum $\u_0 \in L^2_\sigma$. The improvement relative to Theorem \ref{t:Relaxed} is, however, so marginal that we omit the details.
\end{remark}

\section{The proof of Theorem \ref{t:Relaxed} and corollaries} \label{s:The proof of relaxed theorem}

\subsection{The proof of Theorem \ref{t:Relaxed}}
For the proof of Theorem \ref{t:Relaxed} we formulate a lemma which was used (even if not stated explicitly) in~\cite{GKL23}.

\begin{lemma} \label{l:GKL23}
Suppose $X$ is a normed space and $\mathcal{T}$ is a topology on $X$. Suppose that $\{\tau_a\}_{a \in \R}$ is a one-parameter group of isometries $X \to X$. If a sequentially $\mathcal{T}$-closed set $Z \subset X$ contains a ball $\bar{B}(x,r)$, $\tau_a x \to 0$ in $\mathcal{T}$ as $a \to \infty$ and $\tau_a Z \subset Z$ for every $a \in \R$, then $\bar{B}(0,r) \subset Z$.
\end{lemma}

\begin{proof}
Let $\norm{y}_X \leq r$. For every $a \in \R$ we have $x+\tau_{-a} y \in \bar{B}(x,r) \subset Z$ and, therefore, $y + \tau_a x = \tau_a (x + \tau_{-a} y) \in Z$. By the assumptions, $y \in Z$.
\end{proof}

Typically, and below, $\tau_a$ are translations $\tau_a \u_0(x) \defeq \u_0(x-ax_0)$, $x_0 \in \R^3 \setminus \{0\}$. Lemma \ref{l:GKL23} says, informally, that properties that are stable under translations (of $\R^3$) and weak convergence can be transferred from a ball $\bar{B}(\u_0,\epsilon) \subset X$ to $\bar{B}(0,\epsilon) \subset X$. It is important to note that translations of both $\R^3$ and $X$ are at play here. The assumptions on $\mathcal{T}$ and $\{\tau_a\}_{a \in \R}$ are similar in spirit to those of a \emph{group of dislocations} in concentration compactness theory~\cite{ST02}.

\begin{proof}[Proof of \eqref{i:Non-meagre} $\Leftrightarrow$ \eqref{i:A priori}] The direction "$\Leftarrow$" being obvious, we prove the converse. Suppose that for some $\tilde{\nu} > 0$ and a non-meagre set $Y \subset L^2_\sigma \cap Z$ there exists $\u^{\tilde{\nu}} \in \mathcal{W}_{\tilde{\nu}}(\u_0) \cap L^r(0,\infty;X)$. We write $Y = \cup_{M=1}^\infty Y_M$, where each set
\begin{align*}
Y_M
&\defeq \{\u_0 \in L^2_\sigma \cap Z: \exists \u^{\tilde{\nu}} \in \mathcal{W}_{\tilde{\nu}}(\u_0) \cap L^r(0,\infty;X) \text{ with } \\
& \norm{\u^{\tilde{\nu}}}_{L^\infty(0,\infty;L^2)} + \norm{\nabla \u^{\tilde{\nu}}}_{L^2(0,\infty;L^2)} + \norm{\u^{\tilde{\nu}}}_{L^r(0,\infty;X)} \leq M\}
\end{align*}
is weakly$^*$ sequentially closed by Proposition \ref{p:Weak* stability}. By the definition of a meagre set, some $Y_M$ contains a closed ball $\bar{B}(\vv_0,\epsilon) \subset L^2_\sigma \cap Z$. By Proposition \ref{p:Weak* stability}, $Y_M$ satisfies the assumptions of Lemma \ref{l:GKL23} for $\tau_a \u_0(x) \defeq \u_0(x - ae)$, $e \in \R^3 \setminus \{0\}$. Thus $\bar{B}(0,\epsilon) \subset Y_M$.

Fix now $\u_0 \in L^2_\sigma \cap Z$. We set $\lambda \defeq \max\{(\norm{\u_0}_{L^2}/\epsilon)^2, (\norm{\u_0}_Z/\epsilon)^2\}$ and denote $\tilde{\u}_0 \defeq \lambda \u_0(\lambda \cdot)$, so that $\max\{\|\tilde{\u}_0\|_{L^2},\|\tilde{\u}_0\|_Z\} = \epsilon$. We choose $\u \in \W_{\tilde{\nu}}(\tilde{\u}_0)$ with $\norm{\u}_{L^\infty(0,\infty;L^2)} + \norm{\nabla \u}_{L^2(0,\infty;L^2)} + \norm{\u}_{L^r(0,\infty;X)} \leq M$, so that the mapping $\u_\lambda(x,t) \defeq \u(\frac{x}{\lambda},\frac{t}{\lambda^2})/\lambda$ satisfies $\u_\lambda \in \W_{\tilde{\nu}}(\u_0)$ with
\[\norm{\u_\lambda}_{L^\infty(0,\infty;L^2)} + \norm{\nabla \u_\lambda}_{L^2(0,\infty;L^2)} \leq C_{M,\epsilon} \max \left\{ \norm{\u_0}_{L^2}, \norm{\u_0}_Z \right\}\]
and
\begin{align*}
\norm{\u_\lambda}_{L^r(0,\infty;X)}
&= \lambda^{\frac{2}{r} + \frac{3}{s} - 1} \norm{\u}_{L^r(0,\infty;X)} \\
&\leq C_{M,\epsilon} \max \left\{\norm{\u_0}_{L^2}^{2 \left( \frac{2}{r} + \frac{3}{s} - 1 \right)}, \norm{\u_0}_Z^{2 \left( \frac{2}{r} + \frac{3}{s} - 1 \right)} \right\}.
\end{align*}

It remains to prove the estimate \eqref{e:Supercritical a priori} for all $\nu > 0$. Let $\nu > 0$ and $\u_0 \in (L^2_\sigma \cap Z) \setminus \{0\}$. Denote $\lambda = \frac{\nu}{\tilde{\nu}} > 0$ and $\tilde{\u}_0 \defeq \u_0(\lambda \cdot)$, and select $\u \in \W_{\tilde{\nu}}(\tilde{\u}_0)$ such that $\norm{\u}_{L^\infty(0,\infty;L^2)} + \sqrt{\tilde{\nu}} \norm{\nabla \u}_{L^2(0,\infty;L^2)} \leq C_{M,\epsilon} \max \{\|\tilde{\u}_0\|_{L^2}, \|\tilde{\u}_0\|_Z\}$ and $\norm{\u}_{L^r(0,\infty;X)} \leq C \max\{\|\tilde{\u}_0\|_{L^2}^{2 \left( \frac{2}{r} + \frac{3}{s} - 1 \right)}, \norm{\u_0}_Z^{2 \left( \frac{2}{r} + \frac{3}{s} - 1 \right)}\}$. Now
\[\u_\lambda(x,t) \defeq \u \left( \frac{x}{\lambda}, \frac{t}{\lambda} \right)\]
satisfies $\u_\lambda \in \W_\nu(\u_0)$ with
\begin{align*}
\norm{\u_\lambda}_{L^\infty(0,\infty;L^2)} + \sqrt{\nu} \norm{\nabla \u_\lambda}_{L^2(0,\infty;L^2)}
&= \lambda^{\frac{3}{2}} (\norm{\u}_{L^\infty(0,\infty;L^2)} + \sqrt{\tilde{\nu}} \norm{\nabla}_{L^2(0,\infty;L^2)}) \\
&\leq C \lambda^{\frac{3}{2}} \max \left\{\|\tilde{\u}_0\|_{L^2}, \|\tilde{\u}_0\|_Z \right\} \\
&= C \max \left\{ \norm{\u_0}_{L^2}, \norm{\u_0}_Z \right\}
\end{align*}
 and
\begin{align*}
\norm{\u_\lambda}_{L^r(0,\infty;X)} &= \lambda^{\frac{1}{r} + \frac{3}{s}} \norm{\u}_{L^r(0,\infty;X)} \\
&\leq C \lambda^{\frac{1}{r} + \frac{3}{s}} \max \left\{ \norm{\tilde{\u}_0}_{L^2}^{2 \left( \frac{2}{r} + \frac{3}{s} - 1 \right)}, \norm{\tilde{\u}_0}_Z^{2 \left( \frac{2}{r} + \frac{3}{s} - 1 \right)} \right\} \\
&= C \left( \frac{\tilde{\nu}}{\nu} \right)^{3 \left( \frac{5}{3r} + \frac{2}{s} - 1 \right)} \max \left\{ \norm{\u_0}_{L^2}^{2 \left( \frac{2}{r} + \frac{3}{s} - 1 \right)}, \norm{\u_0}_Z^{2 \left( \frac{2}{r} + \frac{3}{s} - 1 \right)} \right\},
\end{align*}
which gives \eqref{e:Minimum for energy}--\eqref{e:Supercritical a priori}.
\end{proof}

\begin{proof}[Proof of \eqref{i:A priori} $\Rightarrow$ \eqref{i:Maximum}]
Suppose, by way of contradiction, that $\frac{5}{3r}+\frac{2}{s} < 1$. Given $\u_0 \in (L^2_\sigma \cap Z) \setminus \{0\}$ we fix a sequence $\nu_j \to 0$ and $\u^{\nu_j} \in \W_{\nu_j}(\u_0)$ satisfying \eqref{e:Minimum for energy}--\eqref{e:Supercritical a priori}. After passing to a subsequence, $\u_j \overset{*}{\rightharpoonup} \u$ in $L^\infty(0,\infty;L^2_\sigma)$ and $\u_j \otimes \u_j \overset{*}{\rightharpoonup} S$ in $L^\infty(0,\infty;\mathbf{M})$, and $(\u,S) = (0,0)$ by \eqref{e:Supercritical a priori}. Thus $(0,0)$ solves \eqref{e:Relaxed NS}, which yields a contradiction.
\end{proof}

\begin{proof}[Proof of \eqref{i:Non-meagre} $\Leftarrow$ \eqref{i:Uniform in nu} $\Rightarrow$ \eqref{i:A priori uniform in nu}]
We assume that \eqref{i:Uniform in nu} holds and prove \eqref{i:Non-meagre}. We write the non-meagre set of $\u_0 \in L^2_\sigma \cap Z$ such that $\limsup_{\eta \to 0} \inf_{\u^\nu \in \W_\nu(\u_0)} (\norm{\u^\nu}_{L^\infty(0,\infty;L^2)} + \sqrt{\nu} \norm{\u^\nu}_{L^2(0,\infty:L^2)} + \norm{\u^\nu}_{L^r(0,\infty;X)}) < \infty$ as $\cup_{M=1}^\infty Y_M$, where
\begin{align*}
Y_M
&\defeq \{\u_0 \in L^2_\sigma \cap Z: \forall \nu \in (0,1/M] \, \exists \u^\nu \in \W_\nu(\u_0) \cap L^r(0,\infty;X) \\
&\text{ with } \norm{\u^\nu}_{L^\infty(0,\infty;L^2)} + \sqrt{\nu} \norm{\u^\nu}_{L^2(0,\infty:L^2)} + \norm{\u^\nu}_{L^r(0,\infty;X)} \leq M \}.
\end{align*}
Every $Y_M$ is closed by Proposition \ref{p:Weak* stability}, and therefore some $Y_M$ contains a closed ball $\bar{B}(\vv_0,\epsilon)$, so that \eqref{i:Non-meagre} holds. The direction \eqref{i:Uniform in nu} $\Rightarrow$ \eqref{i:A priori uniform in nu} is proved just like \eqref{e:Minimum for energy}--\eqref{e:Supercritical a priori} in the proof of \eqref{i:Non-meagre} $\Leftrightarrow$ \eqref{i:A priori}.
\end{proof}

\begin{proof}[Proof of \eqref{i:A priori uniform in nu} $\Rightarrow$ \eqref{i:Exponents}]
Clearly \eqref{i:A priori uniform in nu} implies \eqref{i:Non-meagre}. Therefore,
since \eqref{i:Non-meagre} $\Rightarrow$ \eqref{i:Maximum}, we have $\frac{5}{3r} + \frac{2}{s} \geq 1$. If $\frac{5}{3r} + \frac{2}{s} > 1$, we get a contradiction by letting $\tilde{\nu} \to 0$ in \eqref{e:Uniform a priori}.



\end{proof}

\begin{proof}[Proof of \eqref{i:A priori uniform in nu} $\Rightarrow$ \eqref{i:Uniform in nu}]
Since \eqref{i:A priori uniform in nu} implies \eqref{i:Exponents}, \eqref{e:Minimum uniformly}--\eqref{e:Uniform a priori} imply \eqref{i:Uniform in nu}.
\end{proof}


\subsection{The proof of Corollary \ref{c:Relaxed}} \label{ss:The proof of Besov corollary}
We then prove Corollary \ref{c:Relaxed}. We cover the more general case of $L^r(0,\infty;\dot{B}_{r,\infty}^{\beta_r})$, $\frac{5}{3} \leq r \leq 3$, where $\beta_r = \frac{11-3r}{2r}$.

\begin{proof}[Proof of Corollary \ref{c:Relaxed}]
We prove the direction \eqref{i:Besov existence in L2sigma} $\Rightarrow$ \eqref{i:Besov a priori estimate}, the converse being obvious. Suppose that for some $\nu > 0$ and a non-meagre set of data $\u_0 \in L^2_\sigma \cap Z$ there exists $\u \in \NN_\nu(\u_0) \cap \cup_{T>0} L^r(0,\infty;\dot{B}_{r,\infty}^{\beta_r})$. Thus, one of the closed sets $Y_M \defeq \{\u_0 \in L^2_\sigma \cap Z: \exists \u \in \NN_\nu(\u_0) \text{ with } \norm{\u}_{L^r(0,\infty;\dot{B}_{r,\infty}^{\beta_r})} \leq M\}$ contains a closed ball $\bar{B}(\vv_0,\epsilon)$. By using Lemma \ref{l:GKL23} as before and the fact that $L^r(0,T;\dot{B}_{r,\infty}^{\beta_r})$-regular weak solutions satisfy the energy equality, $\bar{B}(0,\epsilon) \subset Y_M$. By Theorem \ref{t:Relaxed}, we obtain the uniform-in-$\nu$ bound $\norm{\u^\nu}_{L^r(0,\infty;\dot{B}_{r,\infty}^{\beta_r})}^3 \leq C \max\{\norm{\u_0}_{L^2}^2,\norm{\u_0}_Z^2\}$.

It remains to show that for every $\u_0 \in L^2_\sigma \cap Z$ there exist $\nu_j \to 0$ such that $\u^{\nu_j}$ specified above converges weakly$^*$ in $L^\infty(0,\infty;L^2_\sigma)$ to a weak solution $\u$ of the Euler equations with $\norm{\u}_{L^r(0,\infty;\dot{B}_{r,\infty}^{\beta_r})}^3 \leq C \max\{\norm{\u_0}_{L^2}^2,\norm{\u_0}_Z^2\}$.

The existence of the weak$^*$ limit $\u \in L^\infty(0,\infty;L^2_\sigma)$ is immediate. We show that $\u$ solves the Euler equations. Given $\varphi \in C_c^\infty(\R^3 \times [0,\infty))$ with $\nabla \cdot \varphi = 0$ it suffices to show, for a subsequence (of an arbitrary subsequence), that $\int_{\R^3 \times [0,\infty)} \u_j \otimes \u_j: \nabla \varphi \to \int_{\R^3 \times [0,\infty)} \u \otimes \u: \nabla \varphi$. Fix $T > 0$ and a compact set $K \subset \R^3$ such that $\tp{supp}(\varphi) \subset K \times [0,T]$, and let $q \in [1,\frac{6r}{3r-5})$. We obtain the uniform bound $\sup_{j \in \N} \norm{(\nabla^T \varphi) \u_j}_{L^r(0,T;B_{r,\infty}^{\beta_r}(K))} < \infty$, and by Corollary \ref{c:Besov}, $B_{r,\infty}^{\beta_r}(K)$ is compactly embedded into $L^q$. By the Aubin-Lions lemma, $(\nabla^T \varphi) \u_j \to (\nabla^T \varphi) \u$ in $L^3(0,T;L^q)$. Consequently, $\u$ is a weak solution of the Euler equations in $\R^3 \times [0,\infty)$. The estimate $\norm{\u}_{L^r(0,\infty;\dot{B}_{r,\infty}^{\beta_r})}^3 \leq C \max\{\norm{\u_0}_{L^2}^2,\norm{\u_0}_Z^2\}$ follows from Lemma \ref{l:Fatou}.

We finish the proof of the corollary by proving the direction \eqref{i:Besov local existence} $\Rightarrow$ \eqref{i:Besov existence in L2sigma}. Assuming \eqref{i:Besov local existence}, we argue as before to find $\epsilon,T,M>0$ such that
\[\bar{B}(0,\epsilon) \subset \{\u_0 \in Z:\exists \u \in \NN_\nu(\u_0) \text{ with } \norm{\u}_{L^r(0,T;\dot{B}_{r,\infty}^{\beta_r})} \leq M\}.\] Iterating this, using the fact that $\norm{\u(t)}_{L^2} \leq \norm{\u_0}_{L^2} \leq \epsilon$ for all $t > 0$ and applying Remark \ref{r:Besov integrability remark} we obtain $\u \in \NN_\nu(\u_0) \cap L^r(0,\infty;\dot{B}_{r,\infty}^{\beta_r})$ for an arbitrary $\u_0 \in \bar{B}(0,\epsilon)$. Here, we obtain a Leray-Hopf solution with datum $\u_0$ by gluing together at $t=T$ the vector fields $\vv \in \NN_\nu(\u_0) \cap L^r(0,T;\dot{B}_{r,\infty}^{\beta_r})$ and $(x,t) \mapsto \w(x,t-T)$, where $\w \in \NN_\nu(\u(T)) \cap L^r(0,T;\dot{B}_{r,\infty}^{\beta_r})$ (see e.g.~\cite[Lemma 2.4]{Gal00}).
\end{proof}

\subsection{$L^r(0,T;L^s)$ integrability and the energy equality}
We then discuss the $L^r(0,T;L^s)$ integrability of Leray-Hopf solutions. First, by following the argument at~\cite[p. 641]{FS09} essentially verbatim one obtains the following result:
\begin{proposition} \label{p:No higher integrability}
Let $\nu > 0$ and suppose $r,s \in (2,\infty]$ satisfy $1 \leq \frac{2}{r}+\frac{3}{s} < \frac{3}{2}$. If $\u_0 \in L^2_\sigma$ and there exists $\u \in \W_\nu(\u_0) \cap \cup_{T > 0} L^r(0,T;L^s)$, then $\u_0 \in \dot{B}_{s,\tilde{r}}^{-2/\tilde{r}}$, where $\frac{2}{\tilde{r}} = \frac{4}{r}+\frac{3}{s}-1$.
\end{proposition}

Note that if $\frac{2}{r}+\frac{3}{s}=1$, then $\tilde{r} = r$; in this case the converse $\u_0 \in L^2_\sigma \cap \dot{B}_{s,r}^{-2/r}$ $\Rightarrow$ $\exists \u \in \NN_\nu(\u_0) \cap \cup_{T>0} L^r(0,T;L^s)$ is proved in~\cite{FS09,FSV09}. Furthermore, whenever $\frac{2}{r}+\frac{3}{s} < \frac{5}{4}$, we have $\frac{2}{\tilde{r}} + \frac{3}{s} < \frac{3}{2}$ and so $\u \in \W_\nu(\u_0) \cap \cup_{T>0} L^r(0,T;L^s)$ only exists for a meagre set of $L^2_\sigma$ data. However, Proposition \ref{p:No higher integrability} leaves open, for instance, the case of $\cup_{T>0} L^4(0,T;L^4)$.

Theorem \ref{t:Relaxed} rules out $L^4(0,T;L^4)$ integrability (and numerous other conditions that appear in the literature such as Shinbrot's condition). We formulate the following more general result:

\begin{corollary} \label{c:Navier-Stokes}
Let $r,s \in (\frac{4}{3},\infty]$ and $\nu_0 > 0$. If $\frac{5}{3r}+\frac{2}{s} < 1$, then $\u \in \W_{\nu_0}(\u_0) \cap \cup_{T>0} L^r(0,T;L^s)$ only exists for a meagre set of data $\u_0 \in L^2_\sigma$.

If $1 \leq \frac{5}{3r}+\frac{2}{s} < 1+\frac{1}{3r}$, then $\u \in \W_{\nu_0}(\u_0) \cap \cup_{T>0} L^r(0,T;L^s)$ exists for a non-meagre set of $L^2_\sigma$ data if and only if for every $\nu > 0$ and $\u_0 \in L^2_\sigma$ there exists $\u \in \W_\nu(\u_0)$ with $\norm{\u}_{L^\infty(0,\infty;L^2)} + \sqrt{\nu} \norm{\nabla \u}_{L^2(0,\infty;L^2)} \leq C \norm{\u_0}_{L^2}$ and
\[\norm{\u}_{L^r(0,\infty;L^s)} \leq C \nu^{-3 \left( \frac{5}{3r} + \frac{2}{s} - 1 \right)}  \norm{\u_0}_{L^2}^{2 \left( \frac{2}{r} + \frac{3}{s} - 1 \right)}.\]
\end{corollary}


\section{The proof of Theorem \ref{t:Main}} \label{s:The proof of main theorem}
We prove the directions \eqref{i:Time weight} $\Rightarrow$ \eqref{i:L2 data} $\Rightarrow$ \eqref{i:L2 cap L3 data} $\Rightarrow$ \eqref{i:H1 mild} $\Rightarrow$ \eqref{i:A priori estimate for L2 data} $\Rightarrow$ \eqref{i:Time weight} in Theorem \ref{t:Main}.

\begin{proof}[Proof of \eqref{i:Time weight} $\Rightarrow$ \eqref{i:L2 data}]
Select $\nu > 0$ and an increasing function $g \colon (0,\infty) \to (0,\infty)$ such that the set $Y \defeq \{\u_0 \in L^2_\sigma: \exists \u \in \NN_\nu(\u_0) \cap \cup_{T > 0} L^r_g(0,T;L^s)\}$ is non-meagre. Then one of the sets
\[Y_M \defeq \{\u_0 \in L^2_\sigma: \exists \u \in \NN_\nu(\u_0) \text{ such that } \norm{\u}_{L^r_g(0,1/M;L^s)} \leq M\}\]
contains a closed ball $\bar{B}(\vv_0,\epsilon)$. We denote $T \defeq 1/M$. Using Lemma \ref{l:GKL23} as before and applying Theorem \ref{t:Galdi19} we get $\bar{B}(0,\epsilon) \subset Y_M$. Our aim is to show that $\NN_\nu(\u_0) \cap \cap_{\tau > 0} L^r(\tau,\infty;L^s) \neq \emptyset$ for every $\u_0 \in \bar{B}(0,\epsilon)$; claim \eqref{i:L2 data} then follows by scaling.

Let $\u_0 \in \bar{B}(0,\epsilon)$. Choose $\u \in \NN_\nu(\u_0)$ such that $\norm{\u}_{L^r_g(0,T;L^s)} \leq M$. Now $\u(\frac{T}{2}) \in \bar{B}(0,\epsilon)$, and so there exists $\u_2 \in \NN_\nu(\u(\frac{T}{2}))$ with $\norm{\u_2}_{L^r_g(0,T;L^s)} \leq M$. Since $\u \in L^r(\frac{T}{2},T;L^s)$, Theorem \ref{t:Weak-strong uniqueness} implies that $\u|_{\R^3 \times (\frac{T}{2},T)} = \u_2(\cdot+\frac{T}{2})|_{\R^3 \times (\frac{T}{2},T)}$. Repeating the process by induction and using the fact that $\u \in \cup_{T > 0} L^r(T,\infty;L^s)$ we conclude that $\u \in \cap_{\tau > 0} L^r(\tau,\infty;L^s)$.
\end{proof}

\begin{remark}
The proof above, with a bit of work and the scaling argument that we used in the proof of \eqref{e:Supercritical a priori}, gives another equivalent condition: there exists an increasing function $g \colon (0,\infty) \to (0,\infty)$ such that for every $\nu>0$ and $\u_0 \in L^2_\sigma$ there exists $\u^\nu \in \NN_\nu(\u_0)$ with
\[\norm{\u^\nu}_{L^r_{g(\nu^5 \norm{\u_0}_{L^2}^{-4} \cdot)}(0,\infty;L^s)} \leq \nu^{1 - \frac{1}{r}}.\]
\end{remark}

\begin{proof}[Proof of \eqref{i:L2 data} $\Rightarrow$ \eqref{i:L2 cap L3 data}]
Let $\u_0 \in L^2_\sigma \cap L^3$. Use Theorem \ref{t:L2 cap L3 data local solvability} to choose $T > 0$ and a unique weak solution  $\u \in L^\infty(0,T;L^2_\sigma) \cap L^2(0,T;\dot{H}^1) \cap L^r(0,T;L^s)$. Next use \eqref{i:L2 data} to select $\vv \in \mathcal{N}(\u_0) \cap \cap_{\tau > 0} L^r(\tau,\infty;L^s)$. By Theorem \ref{t:Weak-strong uniqueness}, $\u = \vv$ in $\R^3 \times [0,T)$.
\end{proof}

\begin{proof}[Proof of \eqref{i:L2 cap L3 data} $\Rightarrow$ \eqref{i:H1 mild}]
Given $\u_0 \in H^1_\sigma$ and $T>0$ we use Theorem \ref{t:Tao 2} to get a mild $H^1$ solution $\u \in L^\infty(0,\tilde{T};H^1) \cap L^2(0,\tilde{T};H^2)$ for some $\tilde{T} > 0$. On the other hand, \eqref{i:L2 cap L3 data} gives $\vv \in \mathcal{N}(\u_0) \cap L^r(0,\infty;L^s)$. By Theorem \ref{t:Weak-strong uniqueness}, $\u = \vv$ in $\R^3 \times [0,\tilde{T})$. Now \eqref{i:H1 mild} follows from Theorem \ref{t:Higher integrability}.
\end{proof}

\begin{proof}[Proof of \eqref{i:H1 mild} $\Rightarrow$ \eqref{i:A priori estimate for L2 data}] Let $\nu > 0$. Given $\u_0 \in L^2_\sigma$ we denote by $\tilde{\mathcal{N}}_\nu(\u_0)$ the non-empty (by Theorem \ref{t:Leray}) set of Leray-Hopf solutions with initial datum $\u_0$ that satisfy the strong energy inequality:
\begin{equation} \label{e:Strong energy inequality}
\frac{1}{2} \norm{\u(t)}_{L^2}^2 + \nu \int_{t'}^t \norm{\nabla \u(\tau)}_{L^2}^2 d\tau \leq \frac{1}{2} \norm{\u(t')}_{L^2}^2
\end{equation}
at every $t > t'$ and a.e. $t' \geq 0$, including $t' = 0$.

Denote $X^s \defeq L^s$ for $3 \leq s < \infty$ and $X^s \defeq \tp{BMO}$. We intend to show that
\begin{align}
& H(\delta,\tau) \defeq \sup_{\norm{\u_0}_{L^2} \leq \delta} \sup_{\u \in \tilde{\mathcal{N}}_\nu(\u_0)} \norm{\u}_{L^r(\tau,\infty;X^s)} < \infty \quad \text{for all } \delta,\tau > 0, \label{e:H1} \\
& \lim_{\delta \to 0} H(\delta,1) = 0. \label{e:H2}
\end{align}
Once \eqref{e:H1}--\eqref{e:H2} are proved, we use the scaling symmetry $\u_\lambda(x,t) = \u(\frac{x}{\lambda},\frac{t}{\lambda^2})/\lambda$ to set $H(\delta,\tau) = H(\frac{\delta}{\tau^{1/4}},1) \eqdef G(\frac{\delta}{\tau^{1/4}})$ in the statement of condition \eqref{i:A priori estimate for L2 data}. Then $G(\delta) = H(\delta,1) \to 0$ as $\delta \to 0$. Therefore, setting $G(0) \defeq 0$, $G$ is continuous at $0$. Continuity of $G$ at every $\delta > 0$ is proved by a simple continuity argument. We then finish the proof by covering other viscosities.

We first intend to prove \eqref{e:H1}--\eqref{e:H2} in the case $r = \infty$, $s = 3$. Given $\delta, \tau > 0$, $\u_0 \in L^2_\sigma$ with $\norm{\u_0}_{L^2} \leq \delta$ and $\u \in \tilde{\mathcal{N}}_\nu(\u_0)$ we set $t' = 0$ in \eqref{e:Strong energy inequality} to get
\[\abs{\left\{t \in (0,\tau): \norm{\nabla \u(t)}_{L^2}^2 > \frac{\norm{\u_0}_{L^2}^2}{\nu \tau} \right\}} < \frac{\tau}{2}.\]
Therefore, there exists a set of times of positive measure $S \subset (\frac{\tau}{2},\tau)$ such that $\norm{\u(t')}_{H^1} \leq (1+\frac{1}{\sqrt{\nu \tau}}) \norm{\u_0}_{L^2}$ for every $t' \in S$. By using Theorem \ref{t:Leray} we may choose $t' \in S$ such that $\u \in C^\infty(\R^3 \times (t'-\gamma,t'+\gamma))$ for some $\gamma > 0$ and $\u(t'+\cdot)$ is a Leray-Hopf solution with initial datum $\u(t')$. By assumption \eqref{i:H1 mild} and Theorem \ref{t:Tao}, there exists an $H^1$ mild solution $\vv \in C^\infty((t',\infty) \times \R^3) \cap L^\infty(t',\infty;H^1)$ with initial datum $\vv(t') = \u(t')$. By Theorem \ref{t:Weak-strong uniqueness}, $\vv = \u$ in $\R^3 \times [t',\infty)$.

Now $\u$ is a smooth $H^1$ mild solution on $[\tau,\infty) \times \R^3$. Therefore, using Theorem \ref{t:Tao} and the embedding $H^1 \hookrightarrow L^3$ we obtain $\norm{\u}_{L^\infty(\tau,\infty;L^3)} \leq C F((1+\frac{1}{\sqrt{\nu \tau}}) \norm{\u_0}_{L^2})$. We conclude that
\[H(\delta,\tau) \leq C F\left( \left(1+\frac{1}{\sqrt{\nu \tau}} \right) \delta \right)\]
for every $\delta,\tau > 0$. Furthermore, $H(\delta,1) \leq C F((1+\frac{1}{\sqrt{\nu}}) \delta) \to 0$ as $\delta \to 0$.

We next prove \eqref{e:H1}--\eqref{e:H2} in the case $r = 2$, $s = \infty$; the general case then follows by interpolation between $L^2(\tau,\infty;\tp{BMO})$ and $L^\infty(\tau,\infty;L^3)$. First note that $\norm{\u}_{L^\infty(\tau,\infty;H^1)} \leq F((1+\frac{1}{\sqrt{\nu \eta}}) \norm{\u_0}_{L^2})$ as above. Thus, applying Theorem \ref{t:Tao 2} repeatedly and the embedding $H^2 \hookrightarrow L^\infty$ we get
\[\norm{\u}_{L^2(\tau,T;L^\infty)} \leq C
\norm{\u}_{L^2(\tau,T;H^2)} \leq C T F\left( \left(1+\frac{1}{\sqrt{\nu \eta}} \right) \right) \norm{\u_0}_{L^2}\]
for any $T > \tau$. By the energy inequality \eqref{e:Energy inequality} we can choose $T \leq C'\frac{\norm{\u_0}_{L^2}^4}{\nu^3}$ such that $\norm{\u(T)}_{\dot{H}^{1/2}} \leq c\nu$, where $c > 0$ is given by Theorem \ref{t:Fujita-Kato}. Thus, by Remark \ref{r:Embeddings}, $\norm{\u}_{L^2(T,\infty;\tp{BMO})} \leq C \nu^{-1/2} \norm{\u(T)}_{\dot{H}^{1/2}} \leq cC \nu^{1/2}$. We conclude that \eqref{e:H1}--\eqref{e:H2} hold.

We have proved the \emph{a priori} estimate \eqref{e:Supercritical a priori} for one viscosity, which we denote by $\nu_0 > 0$. Let now $\nu > 0$ and $\u_0 \in L^2_\sigma$. Denote $\tilde{\u}_0(x) \defeq \u_0(\lambda x)$, where $\lambda = \frac{\nu}{\nu_0}$. Choose $\u \in \NN_{\nu_0}(\tilde{\u}_0)$ such that $\norm{\u}_{L^r(\tau,\infty;L^s)} \leq G(\frac{\|\tilde{\u}_0\|_{L^2}}{\tau^{1/4}})$ for all $\tau > 0$. Now
\[\u_\lambda(x,t) \defeq \u \left( \frac{x}{\lambda}, \frac{t}{\lambda} \right)\]
satisfies $\u_\lambda \in \NN_\nu(\u_0)$ and
\begin{align*}
\norm{\u_\lambda}_{L^r(\tau,\infty;L^s)}
&= \lambda^{\frac{1}{r} + \frac{3}{s}} \norm{\u}_{L^r(\tau/\lambda,\infty;L^s)}
\leq \lambda^{\frac{1}{r} + \frac{3}{s}} G \left( \frac{\norm{\tilde{\u}_0}_{L^2}}{(\tau/\lambda)^\frac{1}{4}} \right) \\
&= \left( \frac{\nu}{\nu_0} \right)^{1 - \frac{1}{r}} G \left( \left( \frac{\nu_0}{\nu} \right)^{\frac{5}{4}} \frac{\norm{\u_0}_{L^2}}{\tau^\frac{1}{4}} \right),
\end{align*}
which completes the proof.
\end{proof}

\begin{proof}[Proof of \eqref{i:A priori estimate for L2 data} $\Rightarrow$ \eqref{i:Time weight}]
Choose $G \in C[0,\infty)$ as in \eqref{i:A priori estimate for L2 data}. We intend to construct an increasing function $g \colon (0,\infty) \to (0,\infty)$ such that for every $\u_0 \in \bar{B}(0,1)$ there exists $\u \in \NN_1(\u_0) \cap L^r_g(0,1;L^s)$.

We define
\[g(t) \defeq \frac{2^{-k}}{G(2^{k/4})}, \qquad 2^{-k} < t \leq 2^{-k+1}\]
for each $k \in \Z$. If $\u_0 \in \bar{B}(0,1)$, choose $\u \in \NN_1(\u_0)$ such that $\norm{\u}_{L^r(\tau,\infty;L^s)} \leq G(\tau^{-1/4})$ for every $\tau > 0$. Then
\[\int_0^1 \norm{\u(t)}_{L^s}^r g(t) \, dt = \sum_{k=1}^\infty \int_{2^{-k}}^{2^{-k+1}} \norm{\u(t)}_{L^s}^r g(t) \, dt \leq \sum_{k=1}^\infty g(2^{-k+1}) G(2^{k/4}) = 1,\]
which completes the proof of Theorem \ref{t:Main}.
\end{proof}

\begin{remark}
By the same ideas, claims \eqref{i:Time weight}--\eqref{i:H1 mild} of Theorem \ref{t:Main} are also equivalent to well-posedness in various other function spaces, but we do not pursue a general formulation here.
\end{remark}

\section{The proof of Theorem \ref{t:Dense to generic}} \label{s:The proof of dense to generic theorem}
We prove Theorem \ref{t:Dense to generic}; the result holds also on the flat torus $\mathbb{T}^3$.

\begin{proof}[Proof of (i)]
Given $\u_0 \in L^2_\sigma$ and $\u \in \NN_\nu(\u_0)$ we denote
\[\mathcal{G}(t) \defeq \norm{\u(t)}_{L^2}^2 + 2 \nu \int_0^t \norm{\nabla \u(\tau)}_{L^2}^2 d\tau\]
for every $t > 0$. Consider the set $Y$ of data $\u_0 \in Z$ for which some $\u \in \mathcal{N}_\nu(\u_0)$ fails to satisfy the energy equality a.e. We write $ Y = \cup_{M,N=1}^\infty Y_{MN}$, where
\[Y_{MN} \defeq \left\{ \u_0 \in Z: \exists \u \in \NN_\nu(\u_0) \text{ with } \Aint_0^N \mathcal{G}(t) \, dt \leq \left( 1 - \frac{1}{M} \right) \norm{\u_0}_{L^2}^2 \right\}.\]
By Proposition \ref{p:Weak* stability}, each $Y_{MN}$ is closed, and so $L^2_\sigma \setminus Y$ is $G_\delta$.
\end{proof}

\begin{proof}[Proof of (ii)]
Denote $Y \defeq \{\u_0 \in Z: \NN_\nu(\u_0) \text{ is not a singleton}\}$. Given $M \in \mathbb{N}$ we metrise the weak topology of the closed ball $\bar{B}_{L^2(0,\infty;\dot{H}^1)}(0,M) \defeq \{\u \in L^2(0,\infty;\dot{H}^1): \norm{\u}_{L^2(0,\infty;\dot{H}^1)} \leq M\}$ by a metric $d_M$. We write $Y = \cup_{M,N=1}^\infty Y_{MN}$, where the sets
'\[Y_{MN} \defeq \{\u_0 \in Z: \exists \u,\vv \in \mathcal{N}_\nu(\u_0) \cap \bar{B}_{L^2(0,\infty;\dot{H}^1)}(0,M) \text{ with } d_M(\u,\vv) \geq 1/N\}\]
are closed by Proposition \ref{p:Weak* stability}.
\end{proof}

\begin{proof}[Proof of Claim \eqref{i:No dissipation anomaly}]
We write
\[Y \defeq \left\{ u_0 \in Z: \liminf_{\nu \to 0} \sup_{\u^\nu \in \NN_\nu(\u_0)} \int_0^t (\norm{\u_0}_{L^2}^2 - \norm{\u^\nu(\tau)}_{L^2}^2) \, d\tau > 0 \text{ for some } t > 0 \right\}\]
as $Y = \cup_{M=1}^\infty Y_M$, where the sets
\begin{align*}
Y_M &\defeq \left\{ \u_0 \in Z: \inf_{0 < \nu \leq 1/M} \sup_{\u^\nu \in \NN_\nu(\u_0)} \int_0^M (\norm{\u_0}_{L^2}^2 - \norm{\u^\nu(t)}_{L^2}^2) \, dt \geq \frac{1}{M} \right\} \\
&= \left\{ \u_0 \in Z: \sup_{0 < \nu \leq 1/M} \inf_{\u^\nu \in \NN_\nu(\u_0)} \Aint_0^M \norm{\u^\nu(t)}_{L^2}^2 dt \leq \norm{\u_0}_{L^2}^2 - \frac{1}{M^2} \right\}
\end{align*}
are closed by Proposition \ref{p:Weak* stability}.

It remains to prove the part of claim \eqref{i:No dissipation anomaly} in parentheses. Let $\u_0 \in Z \setminus Y$. First choose $\nu_j \to 0$ and $\u^{\nu_j} \in \NN_{\nu_j}(\u_0)$ such that
\begin{equation} \label{e:Limit of norms}
\lim_{j \to \infty} \int_0^T \norm{\u^{\nu_j}(\tau)}_{L^2}^2 d\tau = T\norm{\u_0}_{L^2}^2 \quad \text{for all } T > 0.
\end{equation}
Passing to a subsequence, $\u^{\nu_j} \overset{*}{\rightharpoonup} \u$ in $L^\infty(0,\infty;L^2_\sigma)$, where $\norm{\u(t)}_{L^2} \leq \norm{\u_0}_{L^2}$ for all $t > 0$. Now \eqref{e:Limit of norms} implies that $\lim_{j \to \infty} \norm{\u^{\nu_j}-\u}_{L^2(0,T;L^2)} = 0$ for every $T > 0$ and $\norm{\u(t)}_{L^2} = \norm{\u_0}_{L^2}$ a.e. $t > 0$ for the representative $\u \in C_w([0,\infty);L^2_\sigma)$. In particular,
\begin{align*}
&\liminf_{\nu \to 0} \sup_{\u^\nu \in \NN_\nu(\u_0)} \nu \int_0^T \int_0^t \norm{\nabla \u^\nu(\tau)}_{L^2}^2 d\tau \, dt \\
&\leq \liminf_{\nu \to 0} \sup_{\u^\nu \in \NN_\nu(\u_0)} \int_0^T (\norm{\u_0}_{L^2}-\norm{\u(t)}_{L^2}) \, dt = 0
\end{align*}
for all $T > 0$, which easily implies the remaining claim.
\end{proof}




\section{Necessary conditions for global regularity} \label{s:Necessary conditions for global regularity}
We relate Theorem \ref{t:Dense to generic} to the global regularity problem for the Navier-Stokes and Euler equations. For Navier-Stokes, we recall whole space versions for Schwartz initial data in the form stated in~\cite{Fef06}. We start with the positive direction.

\begin{claim}[existence and smoothness] \label{c:Millennium (A)}
Suppose
\begin{equation} \label{e:Schwartz data}
\abs{\partial_x^\alpha \u_0(x)} \leq C_{\alpha K} (1+\abs{x})^{-K} \quad \text{on } \R^3, \text{ for any } \alpha \in \N_0^n \text{ and } K \geq 0
\end{equation}
and $\nabla \cdot \u_0 = 0$. Then there exists $(\u,p) \in C^\infty(\R^3 \times [0,\infty);\R^3 \times \R)$ such that \eqref{e:NS1}--\eqref{e:NS2} hold and $\u \in L^\infty(0,\infty;L^2_\sigma)$.
\end{claim}

For various necessary and sufficient conditions for Claim \ref{c:Millennium (A)}, see~\cite{Tao13}. If $\u$ is as in Claim \ref{c:Millennium (A)}, then, by~\cite[Corollary 11.1]{Tao13}, $\u \in \cap_{T > 0} L^\infty(0,T;H^1_\sigma) \cap L^2(0,T;H^2)$. Then, by~\cite[Theorem 5.4]{Tao13} and~\cite[Theorem B]{BZZ19}, $\u \in \cap_{k \in \N_0} L^\infty(0,\infty;H^k)$ and all the norms $\norm{\nabla^k \u(t)}_{L^2}$, $k \in \N_0$, are eventually decreasing. In particular, by Theorem \ref{t:Weak-strong uniqueness}, if $\u$ is as in Claim \ref{c:Millennium (A)}, then it is a unique Leray-Hopf solution. Theorems \ref{t:Relaxed} and \ref{t:Dense to generic} then imply the following:
\begin{corollary} \label{c:Millennium (A) corollary}
Suppose Claim \ref{c:Millennium (A)} holds. Then, for a residual set $Y \subset L^2_\sigma$ of initial data, the Leray-Hopf solution is unique and satisfies the energy equality at a.e. $t > 0$.
\end{corollary}

We also discuss the Cauchy problem under external forcing,
\begin{equation} \label{e:NS with force}
\partial_t \u + (\u \cdot \nabla) \u - \nu \Delta \u + \nabla p = \F, \quad \nabla \cdot \u = 0, \quad \u(\cdot,0) = \u_0.
\end{equation}
We formulate the negative side from~\cite{Fef06}.
\begin{claim}[breakdown of smoothness] \label{c:Millennium (C)}
There exists an initial datum $\u_0$ satisfying \eqref{e:Schwartz data} and $\nabla \cdot \u_0$, as well as a force $\F$ satisfying
\[\abs{\partial_x^\alpha \partial_t^m \F(x,t)} \leq C_{\alpha m K} (1+\abs{x}+t)^{-K} \quad \text{on } \R^3 \times [0,\infty), \text{ for any } \alpha,m,K,\]
for which there exists no solution $(\u,p) \in C^\infty(\R^3 \times [0,\infty);\R^3 \times \R)$ of \eqref{e:NS with force} with $\u \in L^\infty(0,\infty;L^2_\sigma)$
\end{claim}

In Proposition \ref{p:Forced proposition} we present a necessary condition for Claim \ref{c:Millennium (C)} in the spirit of Corollary \ref{c:Millennium (A) corollary}. For the formulation we recall that if $T \in \R_+ \cup \{\infty\}$, $\u_0 \in L^2_\sigma$ and $\F \in L^1(0,T;L^2)$, then $\u \in C_w([0,T);L^2_\sigma) \cap L^2(0,T;\dot{H}^1)$ is a Leray-Hopf solution of \eqref{e:NS with force} if $\nabla \cdot \u(\cdot,t) = 0$ a.e. $t \in (0,T)$ and
\[\int_0^T \int_{\R^3} [\u \cdot (\partial_t \varphi + \nu \Delta \varphi + \F) + \u \otimes \u \colon \nabla \varphi] \, dx \, dt = - \int_{\R^3} \u_0 \cdot \varphi(\cdot,0) \, dx\]
for all $\varphi \in C_c^\infty([0,T) \times \R^3)$ with $\nabla \cdot \varphi = 0$ and the energy inequality
\[\frac{1}{2} \norm{\u(t)}_{L^2}^2 + \nu \int_0^t \norm{\nabla \u(\tau)}_{L^2}^2 d\tau \leq \frac{1}{2} \norm{\u_0}_{L^2}^2 + \int_0^t \int_{\R^3} \u \cdot \F \, dx \, dt\]
holds for all $t \in (0,T)$. Given initial datum $\u_0 \in L^2_\sigma$ and external force $\F \in L^1(0,T;L^2)$ we denote by $\mathcal{N}(\u_0,\F)$ the set of Leray-Hopf solutions with datum $(\u_0,\F)$.

\begin{proposition} \label{p:Forced proposition}
Let $T \in \R_+ \cup \{\infty\}$. Suppose there exists a non-meagre set of forces $\F \in L^1(0,T;L^2)$ such that the Leray-Hopf solution with datum $(0,\F)$ is not unique. Then there exists an external force $\F \in C_c^\infty(\R^3 \times (0,T))$ such that there exists no smooth finite energy solution with datum $(0,\F)$.
\end{proposition}

\begin{proof}
We begin by noting that whenever $\F \in C_c^\infty(\R^3 \times (0,T))$ and $\u$ is a smooth finite energy solution with data $(0,\F)$, we have $\mathcal{N}(0,\F) = \{\u\}$. Indeed, by~\cite[Corollary 11.1]{Tao13}, $\u \in L^\infty(0,T;H^1) \cap L^2(0,T;H^2)$. Thus, by weak-strong uniqueness under forcing (e.g.~\cite[Theorem 4.2]{Gal00}), $\u$ is the unique Leray-Hopf solution with datum $(0,\F)$.

We now write the set of forces with a non-unique Leray-Hopf solution as a union $\cup_{M,N=1}^\infty Y_{MN}$, where
\begin{align*}
Y_{MN} \defeq
&\{\F \in L^1(0,T;L^2): \exists \u,\vv \in \mathcal{N}(0,\F) \cap \bar{B}_{L^2(0,\infty;\dot{H}^1)}(0,M) \\
&\text{with } d_M(\u,\vv) \geq 1/N\}.
\end{align*}
By assumption, the closure of some $Y_{MN}$ contains a ball $\bar{B}(\F_0,\epsilon)$. By choosing $\F \in \bar{B}(\F_0,\epsilon)$ we argue as in the proof of Theorem \ref{t:Dense to generic} to conclude that $\mathcal{N}(0,\F)$ is not a singleton, and therefore $\mathcal{N}(0,\F)$ does not contain a smooth finite energy solution.
\end{proof}

\begin{corollary} \label{c:Millennium (C) corollary}
If Claim \ref{c:Millennium (C)} fails, then the Leray-Hopf solution with data $(0,\F)$ is unique for a Baire generic force $\F \in L^1(0,\infty;L^2)$.
\end{corollary}

In~\cite{ABC}, Albritton, Bru\'e and Colombo constructed a force $\F \in L^1(0,T;L^2)$ and two suitable Leray-Hopf solutions on $\R^3 \times (0,T)$ with datum $(0,\F)$.

We also mention a corollary of Theorem \ref{t:Dense to generic} on the global regularity problem on the Euler equations, this time in the standard setting of $H^s_\sigma$ data, $s > \frac{5}{2}$.

\begin{corollary} \label{c:Euler}
Let $s > \frac{5}{2}$, and suppose that for every $\u_0 \in H^s_\sigma$ there exists a solution $\u \in L^\infty_{\tp{loc}}([0,\infty);H^s_\sigma)$ of the Euler equations. Then, claim \eqref{i:No dissipation anomaly} of Theorem \ref{t:Dense to generic} holds for a Baire generic datum $\u_0 \in L^2_\sigma$.
\end{corollary}

\begin{proof}
It suffices to show that claim \eqref{i:No dissipation anomaly} holds for every $\u_0 \in H^s_\sigma$. Under the assumption, for every $\u_0 \in H^s_\sigma$, standard energy methods (along the lines of~\cite[pp. 137--139]{RRS16}) give $\lim_{\nu \to 0} \sup_{\u^\nu \in \mathcal{N}_\nu(\u_0)} \norm{\u^\nu-\u}_{L^\infty(0,T;L^2)} = 0$ for all $T > 0$. (Here, $\norm{\u^\nu-\u}_{L^\infty(0,T;L^2)}$ can even be replaced by $\norm{\u^\nu-\u}_{L^\infty(0,T;H^s)}$; see~\cite[Theorem 2.1]{Mas07}). Since $\norm{\u_0}_{L^2}^2 - \norm{\u^\nu(t)}_{L^2}^2 \leq (\norm{\u_0}_{L^2} + \norm{\u^\nu(t)}_{L^2}) \norm{\u_0-\u^\nu(t)}_{L^2}$, the claim follows.
\end{proof}

Corollary \ref{c:Euler} is one rigorous formulation of the contention that global regularity of the Euler equations is incompatible with "typical" occurrence of anomalous energy dissipation. Anomalous energy dissipation has received widespread support from experiments and simulations (see e.g.~\cite{Fri95}), and rigorous anomalous dissipation results have been proved in recent years in, e.g.,~\cite{AV23,BCCDS24,BDL23,BSV23,DEIJ22,EL24}. To the author's knowledge, the case treated here, unforced 3D Navier-Stokes equations, remains open.

\bibliography{Leray-Hopf}
\bibliographystyle{amsplain}
\end{document}